\documentclass{siamltex}
\usepackage{amsfonts}
\usepackage{amsmath}
\usepackage{amssymb}
\usepackage{enumerate}
\usepackage{graphicx}
\usepackage{lscape}
\usepackage{bm}
\usepackage{color}
\usepackage{multicol}
\usepackage{xspace}
\usepackage{algorithm}
\usepackage{subfigure}
\usepackage{amsfonts}
\usepackage{amsmath}
\usepackage{enumerate}
\newtheorem{assumption}[theorem]{Assumption}
\newtheorem{remark}[theorem]{Remark}

\usepackage{xspace}

\newcommand{\figref}[1]{{Figure~\ref{#1}}}

\newcommand{\thmref}[1]{{Theorem~\ref{#1}}}
\newcommand{\lemref}[1]{{Lemma~\ref{#1}}}
\newcommand{\secref}[1]{{Section~\ref{#1}}}
\newcommand{\assref}[1]{{Assumption~\ref{#1}}}
\newcommand{\propref}[1]{{Proposition~\ref{#1}}}

\title{Magnus-type Integrator for the Finite Element Discretization of Semilinear Parabolic non-Autonomous SPDEs Driven by  multiplicative noise}
\author{ Antoine Tambue \footnotemark[1],\footnotemark[4],\footnotemark[5], \footnotemark[6]
\and Jean Daniel Mukam \footnotemark[2],\footnotemark[3] }

\begin{document}
\maketitle
\renewcommand{\thefootnote}{\fnsymbol{footnote}}
\footnotetext[2]{Fakult\"{a}t f\"{u}r Mathematik, Technische Universit\"{a}t Chemnitz, 09126 Chemnitz, Germany}
\footnotetext[3]{\texttt{jean.d.mukam@aims-senegal.org}}
\footnotetext[4]{\texttt{antonio@aims.ac.za}. Corresponding author.}
\footnotetext[1]{Department of Computing Mathematics and Physics,  Western Norway University of Applied Sciences, Inndalsveien 28, 5063 Bergen.
Center for Research in Computational and Applied Mechanics (CERECAM), and Department of Mathematics and Applied Mathematics, University of Cape Town,
7701 Rondebosch, South Africa. The African Institute for Mathematical Sciences(AIMS) of South Africa  }
\renewcommand{\thefootnote}{\arabic{footnote}}

\pagestyle{myheadings}
\thispagestyle{plain}
\markboth{ A.~Tambue  and J. D.~ Mukam }{MANGUS-TYPE INTEGRATOR FOR NON-AUTONOMOUS SPDEs}

\begin{abstract}
This paper aims to investigate numerical approximation of a general second order non-autonomous semilinear parabolic  
stochastic partial differential equation (SPDE) driven by multiplicative noise. 
Numerical approximations of autonomous SPDEs are thoroughly investigated in the literature, while the   non-autonomous 
case is not yet understood. We discretize the non-autonomous SPDE  driven  by multiplicative noise   by the finite element method  in space and 
the Magnus-type integrator in time.
  We  provide a strong convergence proof of the fully discrete scheme toward the mild solution 
  in the root-mean-square $L^2$  norm.  
 The result reveals how the convergence orders in both space and time depend on the regularity of the noise and the initial data. 
In particular, for multiplicative trace class noise we achieve convergence order 
$\mathcal{O}\left(h^2\left(1+\max(0,\ln\left(t_m/h^2\right)\right)+\Delta t^{1/2}\right)$.
   Numerical simulations  to illustrate our theoretical finding are provided. 
 
\end{abstract}
\section{Introduction}
We consider the numerical approximations of  the following semilinear parabolic  non-autonomous  SPDE driven by mutiplicative noise 
\begin{eqnarray}
 \label{model1}
 \left\{\begin{array}{ll}
  dX=[A(t)X+F(t,X)]dt+B(t,X)dW(t),\quad \text{in} \quad \Lambda\times (0,T],\\
  X(0)=X_0,\quad \hspace{4.9cm} \text{in}\quad \Lambda,
 \end{array}
 \right.
 \end{eqnarray}
 in the Hilbert space $L^2(\Lambda)$, where $\Lambda$ is a bounded domain of $\mathbb{R}^d$, $d=1,2,3$ and $T\in(0,\infty)$.
 The family of unbounded linear operators $A(t)$  are not necessarily self-adjoint. Each $A(t)$  is assumed to generate  an analytic semigroup $S_t(s):= e^{A(t)s}$. 
 The nonlinear functions $F $  and $B$   are  respectively the drift and the diffusion parts. Precise assumptions on $A(t)$, $F$ and $B$ to ensure the existence of
 the unique mild solution of \eqref{model1} are given in the next section. 
 The random initial data is denoted by $X_0$. We denote by $(\Omega, \mathcal{F}, \mathbb{P})$ a probability 
 space with a filtration $(\mathcal{F}_t)_{t\in[0,T]}\subset \mathcal{F}$ that fulfills the usual conditions (see \cite[Definition 2.1.11]{Prevot}).
 The noise term $W(t)$ is assumed to be a $Q$-Wiener process defined on a filtered probability space $\left(\Omega, \mathcal{F}, \mathbb{P}, \{\mathcal{F}_t\}_{t\in[0,T]}\right)$, 
 where the covariance operator $Q : H\longrightarrow H$ is assumed to be linear, self adjoint and positive definite.
It is well known \cite{Prevot} that the noise can be represented as 
\begin{eqnarray}
\label{covariance}
W(t,x)=\sum_{i=0}^{\infty}\sqrt{q}_ie_i(x)\beta_i(t),
\end{eqnarray} 
where $(q_i,e_i)_{i\in\mathbb{N}}$ are the eigenvalues and eigenfunctions of the covariance operator $Q$, and $(\beta_i)_{i\in\mathbb{N}}$ 
are independent and identically distributed standard Brownian motions. 
The deterministic counterpart of \eqref{model1} finds  applications in many fields such as quantum fields theory, electromagnetism, 
nuclear physics (see e.g. \cite{Blanes} and references therein). It is worth to mention that models based on SPDEs can offer a more realistic 
representation of the system than models based only on PDEs, due to uncertainty in the input data. In many situations it is very  hard to exhibit explicit solutions of SPDEs.
For instance the following  non-autonomous linear Stratonovich stochastic ordinary differential equation
\begin{eqnarray}
dy=G_0(t)ydt+\sum_{j=1}^dG_j(t)ydW_j(t),\quad y(0)=y_0\in\mathbb{R}^m
\end{eqnarray}
does not have explicit solution (see e.g. \cite{Blanes2,Kloeden}), unless $G_i$ and $G_j$  commute for all $i,j\geq 0$. 
Numerical algorithms  are  therefore  excellent tools to provide good approximations.
Numerical approximations of  \eqref{model1} 
based on  implicit, explicit Euler methods and exponential integrators with $A(t)=A$, where $A$ is  self-adjoint  are thoroughly investigated 
in the literature, see e.g. \cite{Arnulf2,Kruse1,Kovacs1,Xiaojie1,Yan1,Antonio3,Xiaojie2} and the references therein. If we turn our attention
to the case of time independent operator $A(t)=A$, with $A$ not necessary self-adjoint, the list of references become remarkably short, see e.g., \cite{Antonio1,Antjd1}. 
To the best of our knowledge  numerical approximations of  \eqref{model1} with time dependent linear operator $A(t)$ are not yet investigated in the scientific literature,
due to the complexity of the linear operator $A(t)$ and its semigroup $S_t(s):= e^{A(t)s}$.
Our aim in this paper is to fill that gap and propose an explicit numerical scheme to approximate  \eqref{model1}. 
We use the finite element method for  spatial discretization and Magnus-type integrator for temporal  discretization. 
  Magnus-type integrator is based on a truncation of  Magnus expansion, which was first proposed in \cite{Magnus} 
  to represent the solution of non-autonomous homogeneous differential equation in the exponential form. Magnus expansion was further studied in \cite{Blanes2,Blanes1,Blanes}.
  The first numerical method based on magnus expansion was proposed in \cite{Hochbruck1} for deterministic time-dependent homogeneous Schr\"{o}ndinger equation. 
  The study in \cite{Hochbruck1} was extended  in \cite{Ostermann1} for partial differential equation of the following form
\begin{eqnarray}
u'(t)=A(t)u(t)+b(t),\quad 0<t\leq T,\quad u(0)=u_0.
\end{eqnarray}  
We follow \cite{Ostermann1} and apply the Magnus-type integrator method to the semi-discrete problem \eqref{semi1} and obtain the fully discrete scheme \eqref{scheme3}, 
called stochastic Magnus-type integrators (SMTI). We investigate the strong convergence of the new fully discrete scheme toward the exact solution.
Due to the complexity   of  the  linear operator and the corresponding  semi discrete linear operator after  space discretisation,
novel technical estimates are provided to achieve convergence orders comparable of that of  autonomous SPDEs \cite{Antonio1,Kruse1,Antjd1}.
The result indicates how the convergence orders in both space and time depend on the regularity of the initial data and the noise. In particular 
for multiplicative trace class noise, we achieve optimal convergence orders of $\mathcal{O}\left(h^{\beta}+\Delta t^{\min(\beta,1)/2}\right)$, where $\beta$ 
is the regularity's parameter, defined in  \assref{assumption1}. 

  The rest of this paper is organised as follows. Section \ref{wellposed} provides the general setting, the fully discrete scheme and the main result.
 In  Section \ref{convergenceproof} we provide some preparatory results and we present the proof of the main result. Section \ref{experiment} 
provides some numerical experiments to confirm our theoretical result.

\section{Mathematical setting, numerical scheme and main result}
\label{wellposed}
\subsection{Notations and main assumptions }
\label{notation}
Let  $(H,\langle.,.\rangle_H,\Vert .\Vert)$ be a separable Hilbert space.   For a Banach space $U$, we denote by $L^2(\Omega, U)$ the Banach space of all equivalence classes of square-integrable $U$-valued random variables. Let $L(U,H)$ be 
 the space of bounded linear mappings from $U$ to $H$ endowed with the usual  operator norm $\Vert .\Vert_{L(U,H)}$. 
 By  $\mathcal{L}_2(U,H):=HS(U,H)$,
 we  denote the space of Hilbert-Schmidt operators from $U$ to $H$ equipped with the norm 
 \begin{eqnarray}
 \Vert l\Vert^2_{\mathcal{L}_2(U,H)}:=\sum\limits_{i=1}^{\infty}\Vert l\psi_i\Vert^2, \quad  l\in \mathcal{L}_2(U,H),
 \end{eqnarray}
  where $(\psi_i)_{i=1}^{\infty}$ is an orthonormal basis of $U$. Note that this definition is independent of the orthonormal basis of $U$.  
For simplicity, we use the notations $L(U,U)=:L(U)$. and $\mathcal{L}_2(U,U)=:\mathcal{L}_2(U)$. 
 For all $l\in L(U,H)$ and $l_1\in\mathcal{L}_2(U)$ we have $ll_1\in\mathcal{L}_2(U,H)$ and 
\begin{eqnarray}
\label{chow1}
\Vert ll_1\Vert_{\mathcal{L}_2(U,H)}\leq \Vert l\Vert_{L(U,H)}\Vert l_1\Vert_{\mathcal{L}_2(U)}.
\end{eqnarray}
 The space of Hilbert-Schmidt operators from  $Q^{1/2}(H)$ to $H$ is denoted by $L^0_2:=\mathcal{L}_2(Q^{1/2}(H),H)=HS(Q^{1/2}(H),H)$. As usual, $L^0_2$ is equipped with the norm
 \begin{eqnarray}
 \Vert l\Vert_{L^0_2} :=\Vert lQ^{1/2}\Vert_{HS}=\left(\sum\limits_{i=1}^{\infty}\Vert lQ^{1/2}e_i\Vert^2\right)^{1/2}, \quad  l\in L^0_2,
 \end{eqnarray}
where $(e_i)_{i=1}^{\infty}$ is an orthonormal basis  of $H$.
This definition is independent of the orthonormal basis of $H$. For an $L^0_2$- predictable stochastic process $\phi :[0,T]\times \Lambda\longrightarrow L^0_2$ such that
\begin{eqnarray}
\int_0^t\mathbb{E}\Vert \phi Q^{1/2}\Vert^2_{HS}ds<\infty,\quad t\in[0,T],
\end{eqnarray}
the following relation called It\^{o}'s isometry property holds
\begin{eqnarray}
\label{ito}
\mathbb{E}\left\Vert\int_0^t\phi dW(s)\right\Vert^2=\int_0^t\mathbb{E}\Vert \phi\Vert^2_{L^0_2}ds=\int_0^t\mathbb{E}\Vert\phi Q^{1/2}\Vert^2_{HS}ds,\quad t\in[0,T],
\end{eqnarray}
see e.g. \cite[Step 2 in Section 2.3.2]{Prato} or \cite[Proposition 2.3.5]{Prevot}.

In the rest of this paper, we consider $H=L^2(\Lambda)$. To guarantee the existence of a unique mild solution of \eqref{model1} and for the purpose of the 
convergence analysis, we make the following  assumptions.
\begin{assumption}
 \label{assumption1}
The initial data $X_0 : \Omega\longrightarrow H$ is assumed to be measurable and satisfies $X_0\in L^2\left(\Omega , \mathcal{D}\left(\left(-A(0)\right)^{\beta/2}\right)\right)$, $0\leq \beta\leq 2$.
 \end{assumption}
\begin{assumption}
\label{assumption2}
\begin{itemize}
\item[(i)]
As in \cite{Ostermann1,Gonza1,Ostermann2}, we assume that $\mathcal{D}\left(A(t)\right)=D$, $0\leq t\leq T$ and the family of linear operators $A(t) : D\subset H\longrightarrow H$ to be uniformly sectorial on $0\leq t\leq T$, i.e. there exist constants $c>0$ and $\theta\in\left(\frac{1}{2}\pi,\pi\right)$ such that
\begin{eqnarray}
\left\Vert \left(\lambda\mathbf{I}-A(t)\right)^{-1}\right\Vert_{L(L^2(\Lambda))}\leq \frac{c}{\vert \lambda\vert},\quad \lambda\in S_{\theta},
\end{eqnarray}
where $S_{\theta}:=\left\{\lambda\in\mathbb{C} : \lambda=\rho e^{i\phi}, \rho>0, 0\leq \vert \phi\vert\leq \theta\right\}$. 
As in \cite{Ostermann2}, by a standard scaling argument, we assume $-A(t)$ to be invertible with bounded inverse.
\item[(ii)]
 Similarly to \cite{Gonza1,Ostermann2,Ostermann1,Prato}, we require the following Lipschitz conditions:  there exists a positive constant $K_1$ such that
\begin{eqnarray}
\label{conditionB}
\left\Vert \left(A(t)-A(s)\right)(-A(0))^{-1}\right\Vert_{L(H)}&\leq& K_1\vert t-s\vert,\quad s,t\in[0, T],\\
\left\Vert  (-A(0))^{-1}\left(A(t)-A(s)\right)\right\Vert_{L(D,H)}&\leq& K_1\vert t-s\vert,\quad s,t\in[0, T].
\end{eqnarray}
\item[(iii)]
 Since we are dealing with non smooth data, we follow \cite{Praha} and  assume that 
\begin{eqnarray}
\label{domaine}
\mathcal{D}\left(\left(-A(t)\right)^{\alpha}\right)=\mathcal{D}\left(\left(-A(0)\right)^{\alpha}\right),\quad 0\leq t\leq T,\quad 0\leq \alpha\leq 1
\end{eqnarray}
and there exists a positive constant $K_2$ such that for all $u\in \mathcal{D}((-A(0))^{\alpha})$ the following estimate holds uniformly for $t\in[0,T]$
\begin{eqnarray}
\label{equivnorme1}
K_2^{-1}\left\Vert \left(-A(0)\right)^{\alpha}u\right\Vert\leq \left\Vert (-A(t))^{\alpha}u\right\Vert\leq K_2\left\Vert (-A(0))^{\alpha}u\right\Vert.
\end{eqnarray} 
\end{itemize}
\end{assumption}
\begin{remark}
\label{remark1}
As a consequence of \assref{assumption2} (i) and (iii), for all $\alpha\geq 0$ and $\delta\in[0,1]$, there exists a constant $C_1>0$ 
such that the following estimates hold uniformly for all $t\in[0,T]$
\begin{eqnarray}
\label{smooth}
\left\Vert (-A(t))^{\alpha}e^{sA(t)}\right\Vert_{L(H)}&\leq& C_1s^{-\alpha},\quad \quad s>0,\\
\label{smootha}
 \left\Vert(-A(t))^{-\delta}\left(\mathbf{I}-e^{sA(t)}\right)\right\Vert_{L(H)}&\leq& C_1s^{\delta}, \quad\quad\quad s\geq 0,
\end{eqnarray}
see e.g. \cite[(2.1)]{Ostermann2}.
\end{remark}
\begin{proposition}\cite[Theorem 6.1, Chapter 5]{Pazy}
\label{proposition1}
Let $\Delta(T):=\{(t,s) : 0\leq s\leq t\leq T\}$.
Under \assref{assumption2} there exists a unique evolution system \cite[Definition 5.3, Chapter 5]{Pazy} $U : \Delta(T)\longrightarrow L(H)$ such that
\begin{itemize}
\item[(i)] There exists a positive constant $K_0$ such that
\begin{eqnarray}
\Vert U(t,s)\Vert_{L(H)}\leq K_0,\quad 0\leq s\leq t\leq T.
\end{eqnarray}
\item[(ii)] $U(.,s)\in C^1(]s,T] ; L(H))$, $0\leq s\leq T$,
\begin{eqnarray}
\frac{\partial U}{\partial t}(t,s)=-A(t)U(t,s), \quad0\leq s<t\leq T,\\
\Vert A(t)U(t,s)\Vert_{L(H)}\leq \frac{K_0}{t-s},\quad 0\leq s<t\leq T.
\end{eqnarray}
\item[(iii)] $U(t,.)x\in C^1([0,t[ ; H)$, $0<t\leq T$, $x\in\mathcal{D}(A(0))$ and 
\begin{eqnarray}
\frac{\partial U}{\partial s}(t,s)=-U(t,s)A(s)x,\quad 0\leq s\leq t\leq T,\\
 \Vert A(t)U(t,s)A(s)^{-1}\Vert_{L(H)}\leq K_0, \quad 0\leq s\leq t\leq T.
\end{eqnarray}
\end{itemize}
\end{proposition}
 We equip $V_{\alpha}(t) : = \mathcal{D}\left(\left(-A(t)\right)^{\alpha/2}\right)$, $\alpha\in \mathbb{R}$ with the norm $\Vert u\Vert_{\alpha,t} := \Vert (-A(t))^{\alpha/2}u\Vert$. Due to \eqref{domaine}-\eqref{equivnorme1} and for the seek of  ease notations, we simply write $V_{\alpha}$ and $\Vert .\Vert_{\alpha}$.

We follow \cite{Praha} and  assume the nonlinear operator $F$ to satisfy the following Lipschitz condition. 
\begin{assumption}
\label{assumption3}
The nonlinear operator $F : [0,T]\times H\longrightarrow H$ is assumed to be $\beta/2$-H\"{o}lder continuous with respect to the first variable and Lipschitz continuous with respect to the second variable, i.e. there exists a positive constant $K_3$ such that 
\begin{eqnarray}
\Vert F(s,0)\Vert \leq K_3, \quad \Vert F(t, u)-F(s,v)\Vert \leq K_3\left(\vert t-s\vert^{\beta/2}+\Vert u-v\Vert\right),  
\end{eqnarray}
for all $s,t\in[0,T]$ and $u,v\in H$.
\end{assumption}
 \begin{assumption}
 \label{assumption4}
 We assume the diffusion function $B : [0,T]\times H\longrightarrow L^2_0$ to  be $\beta/2$-H\"{o}lder continuous with respect to the first variable and Lipschitz continuous with respect to the second variable, i.e. there exists a positive constant $K_4$ such that
 \begin{eqnarray}
 \Vert B(s,0)\Vert_{L^0_2}\leq K_4,\quad \Vert B(t,u)-B(s,v)\Vert_{L^0_2}\leq K_4\left(\vert t-s\vert^{\beta/2}+\Vert u-v\Vert\right),
 \end{eqnarray}
 for all $s,t\in[0,T]$ and $u,v\in H$.
 \end{assumption}

 The following theorem ensures the existence of a unique mild solution of \eqref{model1}.
\begin{theorem}
\label{theorem1}\cite[Theorem 1.3]{Praha} Let Assumptions \ref{assumption1}, \ref{assumption2} (i)-(ii),  \ref{assumption3} and  \ref{assumption4} 
be fulfilled. Then the non-autonomous SPDE \eqref{model1} has a unique mild solution $X(t)\in L^2\left(\Omega, \mathcal{D}\left((-A(0))^{\beta/2}\right)\right)$, which takes the following form
\begin{eqnarray}
\label{mild1}
X(t)&=&U(t,0)X_0+\int_0^tU(t,s)F(s,X(s))ds+\int_0^tU(t,s)B(s,X(s))dW(s),
\end{eqnarray} 
where $U(t,s)$ is the evolution system of \propref{proposition1}. Moreover, there exists a positive constant $K_5$ such that 
\begin{eqnarray}
\label{borne1}
\sup_{0\leq t\leq T}\Vert X(t)\Vert_{L^2\left(\Omega, \mathcal{D}\left(\left(-A(0)\right)^{\beta/2}\right)\right)}\leq K_5\left(1+\Vert X_0\Vert_{L^2\left(\Omega,\mathcal{D}\left(\left(-A(0)\right)^{\beta/2}\right)\right)}\right).
\end{eqnarray}
\end{theorem} 
To achieve optimal convergence order in space for multiplicative noise when $\beta\in[1,2]$,
we require  the following further assumption, also used in \cite{Kruse1,Arnulf1,Xiaojie2,Antonio1,Antjd1}.
\begin{assumption}
\label{assumption5}
We assume that there exists a positive constant $c_1>0$,  such 
that $B\left(s,\mathcal{D}((-A(0))^{\frac{\beta-1}{2}})\right)\subset HS\left(Q^{1/2}(H),\mathcal{D}\left(\left(-A(0)\right)^{\frac{\beta-1}{2}}\right)\right)$ 
{\small
\begin{eqnarray}
\left\Vert (-A(0))^{\frac{\beta-1}{2}}B(s,v)\right\Vert_{L^0_2}\leq c_1\left(1+\Vert v\Vert_{\beta-1}\right)\quad,  v\in\mathcal{D}\left(\left(-A(0)\right)^{\frac{\beta-1}{2}}\right),\quad s\in[0,T],
\end{eqnarray}
}
 where $\beta$ comes from \assref{assumption1}.
\end{assumption}

\subsection{Fully discrete scheme and main result}
\label{spacediscretization}
For the seek of simplicity, we assume the family of linear operators $A(t)$\footnote{ Indeed the operators $A(t)$ are identified to their $L^2$ realizations
given in \eqref{family} (see \cite{Suzuki}).} to be of second order and has the following form
\begin{eqnarray}
\label{family}
A(t)u=\sum_{i,j=1}^d\frac{\partial}{\partial x_i}\left(q_{ij}(x,t)\frac{\partial u}{\partial x_j}\right)-\sum_{j=1}^dq_j(x,t)\frac{\partial u}{\partial x_j}.
\end{eqnarray}
We require the coefficients $q_{i,j}$ and $q_j$ to be smooth functions of the variable $x\in\overline{\Lambda}$ and H\"{o}lder-continuous with respect to $t\in[0,T]$. We further assume that there exists a positive constant $c$ such that the following  ellipticity condition holds
\begin{eqnarray}
\label{ellip}
\sum_{i,j=1}^dq_{ij}(x,t)\xi_i\xi_j\geq c\vert \xi\vert^2, \quad (x,t)\in\overline{\Lambda}\times [0,T].
\end{eqnarray}

In the abstract form \eqref{model1}, the nonlinear functions $F:H\longrightarrow H$ and $B:H\longrightarrow HS(Q^{1/2}(H), H)$ are defined by
\begin{eqnarray}
\label{nemistekii1}
(F(v))(x)=f(x, v(x)),\quad (B(v)u)(x)=b(x, v(x)).u(x),
\end{eqnarray}
 for all $x\in \Lambda$,  $v\in H$ and $u\in Q^{1/2}(H)$, where $f:\Lambda\times \mathbb{R}\longrightarrow \mathbb{R}$ and $b:\Lambda\times\mathbb{R}\longrightarrow\mathbb{R}$
 are continuously differentiable functions with globally bounded derivatives. 
 
 Under the above assumptions on $q_{ij}$ and $q_j$, it is  well known  that  the family of linear operators defined by \eqref{family} fulfills  \assref{assumption2} (i)-(ii)  
 with $D=H^2(\Lambda)\cap H^1_0(\Lambda)$, see \cite[Section 7.6]{Pazy} or \cite[Section 5.2]{Tanabe}. The above assumptions on $q_{ij}$ and $q_j$ also imply that \assref{assumption2} (iii) is fulfilled, see e.g. \cite[Example 6.1]{Praha} or \cite{Amann,Seely}.
 
 As in \cite{Suzuki,Antonio1}, we introduce two spaces $\mathbb{H}$ and $V$, such that $\mathbb{H}\subset V$,  depending on the boundary conditions for the domain of the operator $-A(t)$ and the corresponding bilinear form. For  Dirichlet  boundary conditions we take 
\begin{eqnarray}
V=\mathbb{H}=H^1_0(\Lambda)=\{v\in H^1(\Lambda) : v=0\quad \text{on}\quad \partial \Lambda\}.
\end{eqnarray}
For Robin  boundary condition and  Neumann  boundary condition, which is a special case of Robin boundary condition ($\alpha_0=0$), we take $V=H^1(\Lambda)$ and
\begin{eqnarray}
\mathbb{H}=\{v\in H^2(\Lambda) : \partial v/\partial v_A+\alpha_0v=0,\quad \text{on}\quad \partial \Lambda\}, \quad \alpha_0\in\mathbb{R}.
\end{eqnarray}
Using  Green's formula and the boundary conditions, we obtain the corresponding bilinear form associated to $-A(t)$  
\begin{eqnarray*}
a(t)(u,v)=\int_{\Lambda}\left(\sum_{i,j=1}^dq_{ij}(x,t)\dfrac{\partial u}{\partial x_i}\dfrac{\partial v}{\partial x_j}+\sum_{i=1}^dq_i(x,t)\dfrac{\partial u}{\partial x_i}v\right)dx, \quad u,v\in V,
\end{eqnarray*}
for Dirichlet boundary conditions and  
\begin{eqnarray*}
a(t)(u,v)=\int_{\Lambda}\left(\sum_{i,j=1}^dq_{ij}(x,t)\dfrac{\partial u}{\partial x_i}\dfrac{\partial v}{\partial x_j}+\sum_{i=1}^dq_i(x,t)\dfrac{\partial u}{\partial x_i}v\right)dx+\int_{\partial\Lambda}\alpha_0uvdx.
\end{eqnarray*}
for Robin  and Neumann boundary conditions. 
Using  G\aa rding's inequality, it holds that there exist two constants $\lambda_0$ and $c_0$ such that
\begin{eqnarray}
a(t)(v,v)\geq \lambda_0\Vert v \Vert^2_{1}-c_0\Vert v\Vert^2, \quad \forall v\in V,\quad t\in[0,T].
\end{eqnarray}
By adding and subtracting $c_{0}u $ on the right hand side of (\ref{model1}), we obtain a new family of linear operators that we still denote by  $A(t)$.
Therefore the  new corresponding   bilinear form associated to $-A(t)$ still denoted by $a(t)$ satisfies the following coercivity property
\begin{eqnarray}
\label{ellip2}
a(t)(v,v)\geq \; \lambda_0\Vert v\Vert_{1}^{2},\;\;\;\;\;\forall v \in V,\quad t\in[0,T].
\end{eqnarray}
Note that the expression of the nonlinear term $F$ has changed as we have included the term $-c_{0}u$
in a new nonlinear term that we still denote by $F$.

The coercivity property (\ref{ellip2}) implies that $A(t)$ is sectorial on $L^2(\Lambda)$, see e.g. \cite{Stig2}. Therefore   $A(t)$ generates an analytic semigroup   $S_t(s)=e^{s A(t)}$  on $L^{2}(\Lambda)$  such that \cite{Henry}
\begin{eqnarray}
S_t(s)= e^{s A(t)}=\dfrac{1}{2 \pi i}\int_{\mathcal{C}} e^{ s\lambda}(\lambda I - A(t))^{-1}d \lambda,\;\;\;\;\;\;\;
\;s>0,
\end{eqnarray}
where $\mathcal{C}$  denotes a path that surrounds the spectrum of $A(t)$.
The coercivity  property \eqref{ellip2} also implies that $-A(t)$ is a positive operator and its fractional powers are well defined and 
  for any $\alpha>0$ we have
\begin{equation}
\label{fractional}
 \left\{\begin{array}{rcl}
         (-A(t))^{-\alpha} & =& \frac{1}{\Gamma(\alpha)}\displaystyle\int_0^\infty  s^{\alpha-1}{\rm e}^{sA(t)}ds,\\
         (-A(t))^{\alpha} & = & ((-A(t))^{-\alpha})^{-1},
        \end{array}\right.
\end{equation}
where $\Gamma(\alpha)$ is the Gamma function (see \cite{Henry}).  
 The domain  of $(-A(t))^{\alpha/2}$  are  characterized in \cite{Suzuki,Stig1,Stig2}  for $1\leq \alpha\leq 2$  with equivalence of norms as follows.
\begin{eqnarray}
\mathcal{D}((-A(t))^{\alpha/2})=H^1_0(\Lambda)\cap H^{\alpha}(\Lambda)\hspace{1cm} 
\text{(for Dirichlet boundary condition)}\nonumber\\
\mathcal{D}(-A(t))=\mathbb{H},\quad \mathcal{D}((-A(t))^{1/2})=H^1(\Lambda)\hspace{0.5cm} \text{(for Robin boundary condition)}\nonumber\\
\Vert v\Vert_{H^{\alpha}(\Lambda)}\equiv \Vert ((-A(t))^{\alpha/2}v\Vert:=\Vert v\Vert_{\alpha},\quad \forall v\in \mathcal{D}((-A(t))^{\alpha/2}).\nonumber
\end{eqnarray}
The characterization of $\mathcal{D}((-A(t))^{\alpha/2})$ for $0\leq \alpha<1$ can be found in  \cite[Theorem 2.1 \& Theorem 2.2]{Nambu}.

 Let us now  turn our attention to the space discretization of the problem \eqref{model1}.
We start by splitting the domain $\Lambda$ in finite triangles. Let $\mathcal{T}_h$ be the triangulation with maximal 
length $h$ satisfying the usual regularity assumptions, and $V_h\subset V$ be the space of continuous functions that are piecewise 
linear over the triangulation $\mathcal{T}_h$. We consider the projection $P_h$ from $H=L^2(\Lambda)$ to $V_h$ defined for every $u\in H$ by 
\begin{eqnarray}
\label{proj1}
\langle P_hu, \chi\rangle_H=\langle u,\chi\rangle_H, \quad \phi, \chi \in V_h.
\end{eqnarray}
 For all $t\in[0,T]$, the discrete operator $A_h(t) :V_h\longrightarrow V_h$ is defined by 
 \begin{eqnarray}
 \label{discreteoper}
 \langle A_h(t)\phi,\chi\rangle_H=\langle A(t)\phi,\chi\rangle_H=-a(t)(\phi,\chi), \quad \phi,\chi\in V_h.
 \end{eqnarray}
 The coercivity property \eqref{ellip2}  implies that there exist constants $C_2>0$ and $\theta\in(\frac{1}{2}\pi,\pi)$ such that (see e.g. \cite[(2.9)]{Stig2} or \cite{Suzuki,Henry})
 \begin{eqnarray}
 \label{sectorial1}
 \Vert (\lambda\mathbf{I}-A_h(t))^{-1}\Vert_{L(H)}\leq \frac{C_2}{\vert \lambda\vert},\quad \lambda \in S_{\theta}
 \end{eqnarray}
 holds uniformly for $h>0$ and $t\in[0,T]$. The coercivity condition \eqref{ellip2} implies that for any $t\in[0,T]$, $A_h(t)$ generates an analytic semigroup $S^h_t(s):=e^{sA_h(t)}$, $s\in[0,T]$. The coercivity property \eqref{ellip2} also implies that the smooth properties \eqref{smooth}  and \eqref{smootha} hold for $A_h$ uniformly for $h>0$ and $t\in[0,T]$, i.e. for all $\alpha\geq 0$ and $\delta\in[0,1]$, there exists a positive constant $C_3$ such that the following estimates hold uniformly for $h>0$ and $t\in[0,T]$, see e.g. \cite{Suzuki,Henry}
 \begin{eqnarray}
 \label{smooth2}
 \left\Vert(-A_h(t))^{\alpha}e^{sA_h(t)}\right\Vert_{L(H)}&\leq& C_3s^{-\alpha}, \quad s>0, \\
 \label{smooth1}
  \left\Vert (-A_h(t))^{-\delta}\left(\mathbf{I}-e^{sA_h(t)}\right)\right\Vert_{L(H)}&\leq& C_3s^{\delta}, \quad s\geq 0.
 \end{eqnarray}
 The semi-discrete version of  \eqref{model1} consists of finding $X^h(t)\in V_h$, $t\in[0,T]$ such that $X^h(0):=P_hX_0$ and 
 \begin{eqnarray}
 \label{semi1}
 dX^h(t)=[A_h(t)X^h(t)+P_hF(t,X^h(t))]dt+P_hB(t,X^h(t))dW(t),
 \end{eqnarray}
 for $t\in(0,T]$. 
Let us consider the following linear system of non-autonomous ordinary differential equations (ODEs)
\begin{eqnarray}
\label{diff1}
y'(t)=A(t)y(t), \quad y(0) \quad \text{given}.
\end{eqnarray}

It was shown by Magnus  \cite{Magnus} that the solution of \eqref{diff1} can be represented in the following exponential form
\begin{eqnarray}
\label{expo1}
y(t)=e^{\Theta(t)}y(0),\quad t\geq 0,
\end{eqnarray}
where $\Theta(t)$  called Magnus expansion is given by the following series \cite[(3.28)]{Magnus}
\begin{eqnarray}
\label{expan}
\Theta(t)&=&\int_0^tA(\tau)d\tau+\frac{1}{2}\int_0^t\left[A(\tau),\int_0^{\tau}A(\sigma)d\sigma\right]d\tau\nonumber\\
&+&\frac{1}{4}\int_0^t\left[\int_0^{\tau}\left[\int_0^{\sigma}A(\mu)d\mu,A(\sigma)
\right]d\sigma, A(\tau)\right]d\tau\nonumber\\
&+&\frac{1}{12}\int_0^t\left[\int_0^{\tau}A(\sigma)d\sigma, \left[\int_0^{\tau}A(\mu)d\mu, A(\tau)\right]\right]d\tau+\cdots.
\end{eqnarray}
Here the Lie-product $[u,v]$ of $u$ and $v$ is given by $[u,v]=uv-vu$.
For deterministic problems, numerical methods based on this expansion received some attentions since one decade, see e.g. \cite{Blanes,Ostermann1,Hochbruck1,Kaas,Lu}. 
For the time-dependent Schr\"{o}dinger equation \cite{Ostermann1}, the Magnus expansion \eqref{expan} was truncated after the first term and the integral 
was approximated by the mid-point rule. This mid-point rule approximation of $\Theta(t)$ was also used in \cite{Hochbruck1} 
to obtain a second-order Magnus type integrator for non-autonomous deterministic  parabolic partial differential equation (PDE). Note that the convergence analysis in \cite{Ostermann1,Hochbruck1} was only done in time.

Throughout this paper, we take $t_m=m\Delta t\in[0,T]$, where $T=M\Delta t$ for $m, M\in\mathbb{N}$, $m\leq M$.
Motivated by \cite{Ostermann1,Hochbruck1}, we introduce the following fully discrete scheme for \eqref{model1},  called stochastic Magnus-type integrators (SMTI)
\begin{eqnarray}
\label{scheme3}
X^h_{m+1}&=&e^{\Delta tA_{h,m}}X^h_m+\Delta t\varphi_1(\Delta tA_{h,m})P_hF\left(t_{m}, X^h_m\right)\nonumber\\
& +&e^{\Delta tA_{h,m}}P_hB\left(t_{m}, X^h_m\right)\Delta W_m, \quad m=0,\cdots, M,
\end{eqnarray}
$\quad X^h_0=P_hX_0$, where the linear operator $\varphi_1(\Delta t A_{h,m})$ is given by 
\begin{eqnarray}
\varphi_1(\Delta tA_{h,m}):=\dfrac{1}{\Delta t}\int_0^{\Delta t}e^{(\Delta t-s)A_{h,m}}ds,\quad A_{h,m}:=A_h\left(t_{m}\right), 
\end{eqnarray}
and for any $M\in\mathbb{N}$, $\Delta t=T/M$, $t_m=m\Delta t$, $m=0,1,\cdots,M$ and
 \begin{eqnarray}
 \Delta W_m :=W_{(m+1)\Delta t}-W_{m\Delta t}.
 \end{eqnarray}
  Note that the numerical scheme \eqref{scheme3} can be written in the following integral form, useful for the error analysis 
 \begin{eqnarray}
 \label{scheme4}
 X^h_{m+1}&=&e^{\Delta tA_{h,m}}X^h_m+\int_{t_m}^{t_{m+1}}e^{(t_{m+1}-s)A_{h,m}}P_hF\left(t_{m}, X^h_m\right)ds\nonumber\\
 &+&\int_{t_m}^{t_{m+1}}e^{\Delta tA_{h,m}}P_hB\left(t_{m}, X^h_m\right)dW(s).
 \end{eqnarray}
We also note that an equivalent formulation of  the numerical scheme \eqref{scheme3}, easy for simulation is given by
{\small
\begin{eqnarray}
\label{scheme5}
X^h_{m+1}&=&X^h_m+P_hB\left(t_{m}, X^h_m\right)\Delta W_m\nonumber\\
&+&\Delta t\varphi_1(\Delta tA_{h,m})\left[A_{h,m}\left\{X^h_m+P_hB\left(t_{m}, X^h_m\right)\Delta W_m\right\}+P_hF\left(t_{m}, X^h_m\right)\right].
\end{eqnarray}
}

With the numerical method in hand, we can now state its strong convergence result toward the exact solution, which is in fact our main result.
In the rest of this paper  $C$ is a generic constant independent of $h$, $m$, $M$ and $\Delta t$ that may change from one place to another. 
\begin{theorem}\textbf{[Main result]}
\label{mainresult1}
Let Assumptions  \ref{assumption1}, \ref{assumption2}, \ref{assumption3} and \ref{assumption4} be fulfilled. 
\begin{itemize}
\item[(i)] If $0<\beta<1$, then the following error estimate holds
\begin{eqnarray}
\label{main1}
 \left(\mathbb{E}\Vert X(t_m)-X^h_m\Vert^2\right)^{1/2}\leq C\left(h^{\beta}+\Delta t^{\beta/2}\right).
\end{eqnarray}
\item[(ii)] If $1\leq \beta<2$ and moreover  if \assref{assumption5} is satisfied, then the following error estimate holds
\begin{eqnarray}
\label{main2}
\left(\mathbb{E}\Vert X(t_m)-X^h_m\Vert^2\right)^{1/2}\leq C\left(h^{\beta}+\Delta t^{1/2}\right).
\end{eqnarray}
\item[(iii)] If $\beta=2$ and if \assref{assumption5} is fulfilled, then the following error estimate holds
\begin{eqnarray}
\label{main3}
\left(\mathbb{E}\Vert X(t_m)-X^h_m\Vert^2\right)^{1/2}\leq C\left[h^2\left(1+\max(0,\ln(t_m/h^2)\right)+\Delta t^{1/2}\right].
\end{eqnarray}
\end{itemize}
\end{theorem}

\section{Proof of the main result}
\label{convergenceproof}
The proof of the main result needs some preparatory results.
\subsection{Preparatory results}
\label{prepa}
The following lemma will be useful in our convergence proof.
\begin{lemma} \cite{Antjd2}
\label{lemma0}
Let \assref{assumption2} be fulfilled. Then for any $\gamma\in[0,1]$, the following estimates hold uniformly in $h>0$ and $t\in[0,T]$
\begin{eqnarray}
\label{equidiscrete1}
K^{-1}\Vert (-(A_h(0))^{-\gamma}v\Vert\leq \Vert ((-A_h(t))^{-\gamma}v\Vert\leq K\Vert ((-A_h(0))^{-\gamma}v\Vert,\quad v\in V_h,\\
\label{equidiscrete2}
K^{-1}\Vert (-(A_h(0))^{\gamma}v\Vert\leq \Vert ((-A_h(t))^{\gamma}v\Vert\leq K\Vert ((A_h(0))^{\gamma}v\Vert,\quad v\in V_h,
\end{eqnarray}
  where $K$ is a positive constant independent of $t$ and $h$.
\end{lemma}
\begin{lemma}\cite{Antjd2}
\label{lemma0a} 
Under \assref{assumption2}, the following estimates hold
{\small
\begin{eqnarray}
\label{ref3}
\Vert (A_h(t)-A_h(s))(-A_h(r))^{-1}u^h\Vert&\leq& C\vert t-s\vert\Vert u^h\Vert,\quad r,s,t\in[0,T],\quad u^h\in V_h,\\
\label{ref2}
\Vert (-A_h(r))^{-1}\left(A_h(s)-A_h(t)\right)u^h\Vert&\leq& C\vert s-t\vert\Vert u^h\Vert,\quad r,s,t\in[0,T],\quad u^h\in V_h.
\end{eqnarray}
}
\end{lemma}

\begin{remark}
\label{evolutionremark}
From \lemref{lemma0a} and the fact that $\mathcal{D}(A_h(t))=\mathcal{D}(A_h(0))$, it follows from \cite[Theorem 6.1, Chapter 5]{Pazy} that there exists a unique evolution system $U_h :\Delta(T)\longrightarrow L(H)$, satisfying \cite[(6.3), Page 149]{Pazy}
\begin{eqnarray}
\label{ref6}
U_h(t,s)=S^h_s(t-s)+\int_s^tS^h_{\tau}(t-\tau)R^h(\tau,s)d\tau,
\end{eqnarray}
where $S^h_s(t):=e^{A_h(s)t}$, $R^h(t,s):=\sum\limits_{m=1}^{\infty}R^h_m(t,s)$, with $R^h_m(t,s)$ satisfying the following  recurrence relation \cite[(6.22), Page 153]{Pazy}
\begin{eqnarray}
R^h_{m+1}=\int_s^tR^h_1(t,s)R^h_m(\tau,s)d\tau,\quad m\geq 1
\end{eqnarray}
and $R^h_1(t,s):=(A_h(s)-A_h(t))S^h_s(t-s)$. Note also that from \cite[(6.6), Chpater 5, Page 150]{Pazy}, the following identity holds 
\begin{eqnarray}
\label{ref7}
R^h(t,s)=R_1^h(t,s)+\int_s^tR_1^h(t,\tau)R^h(\tau,s)d\tau.
\end{eqnarray} 
The mild solution of \eqref{semi1} is therefore given by
\begin{eqnarray}
\label{mild4}
X^h(t)&=&U_h(t,0)P_hX_0+\int_0^tU_h(t,s)P_hF\left(s,X^h(s)\right)ds\nonumber\\
&+&\int_0^tU_h(t,s)P_hB\left(s,X^h(s)\right)dW(s).
\end{eqnarray}
\end{remark}
\begin{lemma}
\label{pazylemma}
Under \assref{assumption2}, the  evolution system $U_h :\Delta(T)\longrightarrow H$ satisfies the following
\begin{itemize}
\item[(i)] $U_h(.,s)\in C^1(]s,T]; L(H))$, $0\leq s\leq T$ and 
\begin{eqnarray}
\frac{\partial U_h}{\partial t}(t,s)=-A_h(t)U_h(t,s), \quad 0\leq s\leq t\leq T,\\
\Vert A_h(t)U_h(t,s)\Vert_{L(H)}\leq \frac{C}{t-s},\quad 0\leq s<t\leq T.
\end{eqnarray}
\item[(ii)] $U_h(t,.)u\in C^1([0,t[; H)$, $0<t\leq T$, $u\in\mathcal{D}(A_h(0))$ and 
\begin{eqnarray}
\frac{\partial U_h}{\partial s}(t,s)u=-U_h(t,s)A_h(s)u,\quad 0\leq s\leq t\leq T,\\
 \Vert A_h(t)U_h(t,s)A_h(s)^{-1}\Vert_{L(H)}\leq C, \quad 0\leq s\leq t\leq T.
\end{eqnarray}
\end{itemize}
\end{lemma}
\begin{proof}
The proof is similar  to that of \cite[Theorem 6.1, Chapter 5]{Pazy} using \eqref{smooth1}, \eqref{smooth2}, Lemmas \ref{lemma0a} and \ref{lemma0}.
\end{proof}

\begin{lemma}\cite{Antjd2}
\label{evolutionlemma}
Let  \assref{assumption2} be fulfilled. 
\begin{itemize}
\item[(i)] The following estimates hold
\begin{eqnarray}
\label{reste1}
\Vert R^h_1(t,s)\Vert_{L(H)}\leq C,&&\quad
\Vert R^h_m(t,s)\Vert_{L(H)}\leq \frac{C}{m!}(t-s)^{m-1},\quad m\geq 1,\\
\label{reste2}
\Vert R^h(t,s)\Vert_{L(H)}\leq C,&&\quad \Vert U_h(t,s)\Vert_{L(H)}\leq C,\quad 0\leq s\leq t\leq T.
\end{eqnarray}
\item[(ii)] For any $0\leq\alpha\leq 1$, $0\leq\gamma\leq 1$ and $0\leq s\leq t\leq T$, the following estimates holds
\begin{eqnarray}
\label{ae1}
\Vert (-A_h(r))^{\alpha}U_h(t,s)\Vert_{L(H)}&\leq& C(t-s)^{-\alpha},\quad r\in[0,T],\\
\label{ae3}
\Vert U_h(t,s)(-A_h(r))^{\alpha}\Vert_{L(H)}&\leq& C(t-s)^{-\alpha},\quad r\in[0,T],\\
\label{ae2}
 \Vert (-A_h(r))^{\alpha}U_h(t,s)(-A_h(s))^{-\gamma}\Vert_{L(H)}&\leq& C(t-s)^{\gamma-\alpha}, \quad r\in[0,T].
\end{eqnarray}
\item[(iii)] For any $0\leq s\leq t\leq T$ the following useful estimates hold 
\begin{eqnarray}
\label{hen1}
\Vert \left(U_h(t,s)-\mathbf{I}\right)(-A_h(s))^{-\gamma}\Vert_{L (H)}&\leq& C(t-s)^{\gamma}, \quad 0\leq \gamma\leq 1,\\
\label{hen2}
\Vert \left (-A_h(r))^{-\gamma}(U_h(t,s)-\mathbf{I}\right)\Vert_{L (H)}&\leq& C(t-s)^{\gamma}, \quad 0\leq \gamma\leq 1.
\end{eqnarray}
\end{itemize}
\end{lemma}
The following space and time regularity of the semi-discrete problem \eqref{semi1} will be useful in our convergence analysis.
\begin{lemma} 
\label{regularitylemma}
Let  Assumptions \ref{assumption1}, \ref{assumption2} (i)-(ii), \ref{assumption3} and \ref{assumption4} be fulfilled with
the corresponding  $0\leq \beta<1$. Then for all $\gamma\in[0,\beta]$ the following estimates hold
\begin{eqnarray}
\label{regular1}
\Vert (-A_h(r))^{\gamma/2}X^h(t)\Vert_{L^2(\Omega,H)}&\leq& C,\quad\hspace{2cm} 0\leq r,t\leq T,\\
\label{regular2}
\Vert X^h(t_2)-X^h(t_1)\Vert_{L^2(\Omega,H)}&\leq& C(t_2-t_1)^{\beta/2}, \quad 0\leq t_1\leq t_2\leq T.
\end{eqnarray}
Moreover if \assref{assumption5} is fulfilled, then \eqref{regular1} and \eqref{regular2} hold for $\beta=1$.
\end{lemma}
\begin{proof}
 We first show that $\sup\limits_{t\in[0,T]}\Vert X^h(t)\Vert^2_{L^2(\Omega, H)}\leq C$.
 Taking the norm in both side of \eqref{mild4} and using the  inequality $(a+b+c)^2\leq 3a^2+3b^2+3c^2$, $a, b, c\in\mathbb{R}_+$ yields
 {\small
\begin{eqnarray}
\label{regular3}
\Vert X^h(t)\Vert^2_{L^2(\Omega,H)}&\leq& 3\Vert U_h(t,0)P_hX_0\Vert^2_{L^2(\Omega,H)}+3\left\Vert\int_0^tU_h(t,s)P_hF\left(s,X^h(s)\right)ds\right\Vert^2_{L^2(\Omega,H)}ds
\nonumber\\
&+&3\left\Vert\int_0^tU_h(t,s)P_hB\left(s,X^h(s)\right)dW(s)\right\Vert^2_{L^2(\Omega,H)}:=I_0+I_1+I_2.
\end{eqnarray}
}
Using \lemref{evolutionlemma} (i) and the uniformly boundedness of $P_h$, it holds that
\begin{eqnarray}
\label{regular4}
I_0\leq 3\Vert X_0\Vert^2_{L^2(\Omega,H)}\leq C.
\end{eqnarray}
Using again \lemref{evolutionlemma} (i), \assref{assumption3} and the uniformly boundedness of $P_h$, it holds that
{\small
\begin{eqnarray}
\label{regular5}
I_1\leq 3\left(\int_0^t\Vert U_h(t,s)P_hF\left(s,X^h(s)\right)\Vert_{L^2(\Omega,H)}\right)^2\leq C\left(\int_0^t\left(C+\Vert X^h(s)\Vert_{L^2(\Omega,H)}\right)ds\right)^2\nonumber.
\end{eqnarray}
}
Using H\"{o}lder inequality yields
\begin{eqnarray}
\label{regular6}
I_1\leq C+C\int_0^t\Vert X^h(s)\Vert^2_{L^2(\Omega,H)}ds.
\end{eqnarray}
Applying the it\^{o}-isometry's property, using \lemref{evolutionlemma} (i) and \assref{assumption4}, it holds that
{\small
\begin{eqnarray}
\label{regular7}
I_2=3\int_0^t\Vert U_h(t,s)P_hB\left(s,X^h(s)\right)\Vert^2_{L^0_2}ds\leq C+C\int_0^t\Vert X^h(t)\Vert^2_{L^2(\Omega,H)}ds.
\end{eqnarray}
}
Substituting \eqref{regular7}, \eqref{regular6} and \eqref{regular4} in \eqref{regular3} yields
\begin{eqnarray}
\label{regular8}
\Vert X^h(t)\Vert^2_{L^2(\Omega,H)}\leq C+C\int_0^t\Vert X^h(s)\Vert^2_{L^2(\Omega,H)}ds.
\end{eqnarray}
Applying the continuous Gronwall's lemma to \eqref{regular8} yields
\begin{eqnarray}
\label{regular8a}
\Vert X^h(t)\Vert^2_{L^2(\Omega,H)}\leq C,\quad t\in[0,T].
\end{eqnarray}
Let us now prove \eqref{regular1}. 
Pre-multiplying \eqref{mild4} by $(-A_h(r))^{\gamma/2}$, taking  the norm in both sides  and using triangle inequality  yields
{\small
\begin{eqnarray}
\label{regular9}
\left\Vert (-A_h(r))^{\gamma/2}X^h(t)\right\Vert_{L^2(\Omega,H)}&\leq& \left\Vert (-A_h(r))^{\gamma/2}U_h(t,0)P_hX_0\right\Vert_{L^2(\Omega,H)}\nonumber\\
&+&\int_0^t\left\Vert (-A_h(r))^{\gamma/2}U_h(t,s)P_hF\left(s,X^h(s)\right)\right\Vert_{L(\Omega,H)}ds
\nonumber\\
&+&\left\Vert\int_0^t(-A_h(r))^{\gamma/2}U_h(t,s)P_hB\left(s,X^h(s)\right)dW(s)\right\Vert_{L^2(\Omega,H)}\nonumber\\
&:=&II_0+II_1+II_2.
\end{eqnarray}
}
Inserting $(-A_h(0))^{-\gamma/2}(-A_h(0))^{\gamma/2}$, using  \lemref{evolutionlemma} (ii) and \lemref{lemma0}, it holds that
{\small
\begin{eqnarray}
\label{premier}
II_0\leq \Vert (-A_h(r))^{\gamma/2}U_h(t,0)(-A_h(0))^{-\gamma/2}\Vert_{L(H)}\Vert (-A_h(0))^{\gamma/2}X_0\Vert\leq C.
\end{eqnarray}
}
Using  Lemmas \ref{lemma0}, \ref{evolutionlemma} (ii), \assref{assumption3} and \eqref{regular8a} yields
\begin{eqnarray}
\label{deuxieme}
II_1&\leq& C\int_0^t\left\Vert(-A_h(s))^{\gamma/2}U_h(t,s)\right\Vert_{L(H)}\sup_{t\in[0,T]}\left\Vert F\left(s,X^h(s)\right)\right\Vert ds\nonumber\\
&\leq& C\sup_{s\in[0,T]}\left(1+\Vert X^h(s)\Vert_{L^2(\Omega,H)}\right)\int_0^t(t-s)^{-\gamma/2}ds\leq C.
\end{eqnarray}
Applying the It\^{o}-isometry property, using  Lemmas \ref{lemma0}, \ref{evolutionlemma} (ii), \assref{assumption4} and \eqref{regular8a} yields
\begin{eqnarray}
\label{troisieme}
II_2^2&=&\int_0^t\left\Vert (-A_h(0))^{\gamma/2}U_h(t,s)P_hB\left(s,X^h(s)\right)\right\Vert^2_{L^0_2}ds\nonumber\\
&\leq& C\sup_{s\in[0,T]}\left(1+\Vert X^h(s)\Vert^2_{L^2(\Omega,H)}\right)\int_0^t(t-s)^{-\gamma}ds\leq C.
\end{eqnarray}
Substituting \eqref{troisieme}, \eqref{deuxieme} and \eqref{premier} in \eqref{regular9} completes the proof of \eqref{regular1}.   The proof of \eqref{regular2} follows from \eqref{mild4}. In fact from \eqref{mild4} we have
\begin{eqnarray}
\label{inter1}
\Vert X^h(t_2)-X^h(t_1)\Vert_{L^2(\Omega,H)}&\leq& \left\Vert \left(U_h(t_2,0)-U_h(t_1,0)\right)P_hX_0\right\Vert_{L^2(\Omega,H)}\nonumber\\
&+&\int_0^{t_1}\left\Vert \left(U_h(t_2,s)-U_h(t_1,s)\right)P_hF\left(s,X^h(s)\right)\right\Vert_{L^2(\Omega,H)} ds\nonumber\\
&+&\int_{t_1}^{t_2}\left\Vert U_h(t_2,s)P_hF\left(s,X^h(s)\right)\right\Vert_{L^2(\Omega,H)} ds\nonumber\\
&+&\left\Vert\int_0^{t_1}U_h(t_2,s)-U_h(t_1,s))P_hB\left(s,X^h(s)\right) dW(s)\right\Vert_{L^2(\Omega,H)}\nonumber\\
&+&\left\Vert\int_{t_1}^{t_2} U_h(t_2,s)P_hB\left(s,X^h(s)\right) dW(s)\right\Vert_{L^2(\Omega,H)}\nonumber\\
&:=&III_0+III_1+III_2+III_3+III_4.
\end{eqnarray}
Inserting an appropriate power of $-A_h(t_1)$, using Lemmas \ref{evolutionlemma} (ii)-(iii) and \cite[Lemma 1]{Antjd1} yields
{\small
\begin{eqnarray}
\label{inter2}
III_0&=& \left\Vert (U_h(t_2,t_1)-\mathbf{I})U_h(t_1,0)P_hX_0\right\Vert_{L^2(\Omega,H)}\nonumber\\
&\leq& \left\Vert (U_h(t_2,t_1)-\mathbf{I})(-A_h(t_1))^{-\beta/2}\right\Vert_{L(H)}\nonumber\\
&&\times\left\Vert (-A_h(t_1))^{\beta/2}U_h(t_1,0)(-A_h(t_1))^{-\beta/2}\right\Vert_{L(H)}\left\Vert (-A_h(t_1))^{\beta/2}P_hX_0\right\Vert_{L^2(\Omega,H)}\nonumber\\
&\leq & C(t_2-t_1)^{\beta/2}.
\end{eqnarray}
}
Using \assref{assumption4}, \eqref{regular1},  \lemref{evolutionlemma} (ii) and (iii) yields
{\small
\begin{eqnarray}
\label{inter3}
III_1&\leq&\int_0^{t_1}\left\Vert (U_h(t_2,t_1)-\mathbf{I})U_h(t_1,s)\right\Vert_{L(H)}\left\Vert P_hF\left(s,X^h(s)\right)\right\Vert_{L^2(\Omega,H)}ds\nonumber\\
&\leq& C\int_0^{t_1}\left\Vert (U_h(t_2,t_1)-\mathbf{I})(-A_h(t_1))^{-\beta/2}\right\Vert_{L(H)}\left\Vert (-A_h(t_1))^{\beta/2}U_h(t_1,s)\right\Vert_{L(H)}ds\nonumber\\
&\leq& C\int_0^{t_1}(t_2-t_1)^{\beta/2}(t_1-s)^{-\beta/2}ds\nonumber\\
&\leq& C(t_2-t_1)^{\beta/2}.
\end{eqnarray}
}
Using \lemref{evolutionlemma} (i) and \assref{assumption3}, it holds that
\begin{eqnarray}
\label{inter4}
III_2\leq C\int_{t_1}^{t_2}\sup_{s\in[0,T]}\left\Vert F\left(s,X^h(s)\right)\right\Vert_{L^2(\Omega,H)}ds\leq C(t_2-t_1).
\end{eqnarray}
Using the It\^{o}-isometry property, \assref{assumption5}, \eqref{regular1}, Lemma \ref{evolutionlemma} (ii)-(iii) and following the same lines as the estimate of $III_1$ yields
\begin{eqnarray}
\label{inter5}
III_3^2\leq C(t_2-t_1)^{\beta}.
\end{eqnarray}
Using the It\^{o}-isometry property and following the same lines as that of $III_2$ yields
\begin{eqnarray}
\label{inter6}
III_4^2\leq C(t_2-t_1).
\end{eqnarray}
Substituting \eqref{inter6}, \eqref{inter5}, \eqref{inter4}, \eqref{inter3} and \eqref{inter2} in \eqref{inter1} completes the proof of \eqref{regular2}.
\end{proof}

Let us consider the following deterministic  problem: find $u\in V$ such that
\begin{eqnarray}
\label{determ1}
u'=A(t)u,\quad u(\tau)=v,\quad t\in(\tau,T].
\end{eqnarray}
The corresponding semi-discrete problem in space is: find $u_h\in V_h$ such that
\begin{eqnarray}
\label{determ2}
u_h'(t)=A_h(t)u_h,\quad u_h(\tau)=P_hv,\quad t\in(\tau,T],\quad \tau\geq 0.
\end{eqnarray}
Let us define the operator
\begin{eqnarray}
T_h(t,\tau):=U(t,\tau)-U_h(t,\tau)P_h,
\end{eqnarray}
so that $u(t)-u_h(t)=T_h(t,\tau)v$. The following lemma will be useful in our convergence analysis. 
\begin{lemma}\cite{Antjd2}
\label{spaceerrorlemma}
Let $r\in[0,2]$ and  $0\leq\gamma\leq r$. Let  \assref{assumption2} be fulfilled.  Then the following error estimate holds for the semi-discrete approximation  \eqref{determ2}
{\small
\begin{eqnarray}
\label{er0}
\Vert u(t)-u_h(t)\Vert=\Vert T_h(t,\tau)v\Vert\leq Ch^r(t-\tau)^{-(r-\gamma)/2}\Vert v\Vert_{\gamma}, \quad  v\in \mathcal{D}\left(\left(-A(0)\right)^{\gamma/2}\right).
\end{eqnarray}
}
\end{lemma}
\begin{proposition}\textbf{[Space error]}
\label{proposition2}
Let Assumptions  \ref{assumption1}, \ref{assumption2}, \ref{assumption3} and \ref{assumption4} be fulfilled. Let $X(t)$ and $X^h(t)$ be the mild solution of \eqref{model1} and \eqref{semi1} respectively.
\begin{itemize}
\item[(i)] If $0<\beta<1$, 
then the following error estimate holds 
\begin{eqnarray}
\label{time1}
\Vert X(t)-X^h(t)\Vert_{L^2(\Omega,H)}\leq Ch^{\beta},\quad 0\leq t\leq T.
\end{eqnarray}
\item[(ii)] If $1\leq \beta<2$ and moreover if  \assref{assumption5} is fulfilled, then the following error estimate holds
\begin{eqnarray}
\label{time3}
\Vert X(t)-X^h(t)\Vert_{L^2(\Omega,H)}\leq Ch^{\beta},\quad 0\leq t\leq T,
\end{eqnarray}
\item[(iii)] If $\beta=2$ and moreover if  \assref{assumption5} is fulfilled, then the following error estimate holds
\begin{eqnarray}
\label{time4}
\Vert X(t)-X^h(t)\Vert_{L^2(\Omega,H)}\leq Ch^{2}\left(1+\max\left(0,\ln(t/h^2)\right)\right),\, 0< t\leq T.
\end{eqnarray}
\end{itemize}
\end{proposition}

\begin{proof} 
Subtracting \eqref{mild4} form \eqref{mild1},  taking the $L^2$ norm  and using triangle inequality yields
{\small
\begin{eqnarray}
\label{estiI}
\Vert X(t)-X^h(t)\Vert_{L^2(\Omega,H)}&\leq& \left\Vert U(t,0)X_0-U_h(t,0)P_hX_0\right\Vert_{L^2(\Omega,H)}\nonumber\\
&+&\left\Vert\int_0^{t}\left[U(t,s)F\left(s,X(s)\right)-U_h(t,s)P_hF\left(s,X^h(s)\right)\right]
 ds\right\Vert_{L^2(\Omega,H)}\nonumber\\
 &+&\left\Vert\int_0^{t}\left[U(t,s)B\left(s,X(s)\right)-U_h(t,s)P_hB\left(s,X^h(s)\right)\right]dW(s)\right\Vert_{L^2(\Omega,H)}\nonumber\\
 & =:&IV_0+IV_1+IV_2.
\end{eqnarray}
}
Using \lemref{spaceerrorlemma} with $r=\gamma=\beta$  yields
\begin{eqnarray}
\label{jour1}
IV_0\leq Ch^{\beta}\Vert X_0\Vert_{L^2\left(\Omega, \mathcal{D}\left(\left(-A(0)\right)^{\beta/2}\right)\right)}\leq Ch^{\beta}.
\end{eqnarray}
Using \lemref{spaceerrorlemma} with $r=\beta$, $\gamma=0$, \assref{assumption3},  Lemmas \ref{regularitylemma} and \ref{evolutionlemma} yields
\begin{eqnarray}
\label{estiI2}
IV_1&\leq& \int_0^t\left\Vert U(t,s)F\left(s,X(s)\right)-U(t,s)F\left(s,X^h(s)\right)\right\Vert_{L^2(\Omega,H)}ds\nonumber\\
&+&\int_0^t\left\Vert U(t,s)F\left(s,X^h(s)\right)-U_h(t,s)P_hF\left(s,X^h(s)\right)\right\Vert_{L^2(\Omega,H)}ds\nonumber\\
&\leq& C\int_0^t\left\Vert X(s)-X^h(s)\right\Vert_{L^2(\Omega,H)}ds+Ch^{\beta}\int_0^t(t-s)^{-\beta/2}ds\nonumber\\
&\leq& Ch^{\beta}+C\int_0^{t}\left\Vert X(s)-X^h(s)\right\Vert_{L^2(\Omega,H)}ds.
\end{eqnarray}
Using the It\^{o}-isometry property, \lemref{regularitylemma}, \lemref{spaceerrorlemma} with $r=\beta$ and $\gamma=\frac{\beta-1}{2}$ yields
\begin{eqnarray}
\label{estiI3}
IV_2^2&=& \int_0^t\left\Vert U(t,s)B\left(s,X(s)\right)-U_h(t,s)P_hB\left(s,X^h(s)\right)\right\Vert^2_{L^0_2}ds\nonumber\\
&\leq& \int_0^t\left\Vert U(t,s)B\left(s,X(s)\right)-U(t,s)B\left(s,X^h(s)\right)\right\Vert^2_{L^0_2}ds\nonumber\\
&+&\int_0^t\left\Vert U(t,s)B\left(s,X^h(s)\right)-U_h(t,s)P_hB\left(s,X^h(s)\right)\right\Vert_{L^0_2}ds\nonumber\\
&\leq& C\int_0^t\left\Vert X(s)-X^h(s)\right\Vert^2_{L^2(\Omega,H)}ds+Ch^{2\beta}\int_0^t(t-s)^{-1+\beta}ds\nonumber\\
&\leq& Ch^{2\beta}+C\int_0^{t}\left\Vert X(s)-X^h(s)\right\Vert_{L^2(\Omega,H)}ds.
\end{eqnarray}
Substituting \eqref{estiI3}, \eqref{estiI2} and \eqref{jour1} in \eqref{estiI} yields
\begin{eqnarray}
\label{lat1}
\left\Vert X(t)-X^h(t)\right\Vert^2_{L^2(\Omega,H)}&\leq& Ch^{2\beta}+C\int_0^{t}\left\Vert X(s)-X^h(s)\right\Vert^2_{L^2(\Omega,H)}ds.
\end{eqnarray}
Applying the continuous Gronwall's lemma to \eqref{lat1} yields 
\begin{eqnarray}
\label{esti1}
\left\Vert X(t)-X^h(t)\right\Vert_{L^2(\Omega,H)}\leq Ch^{\beta}.
\end{eqnarray}
\end{proof}

For non commutative operators $H_j$ on a Banach space, we introduce the following notation for the composition
\begin{eqnarray}
\prod_{j=l}^kH_j=\left\{\begin{array}{ll}
H_kH_{k-1}\cdots H_l\quad \text{if} \quad k\geq l,\\
\mathbf{I}\quad \hspace{2cm} \text{if} \quad k<l.
\end{array}
\right.
\end{eqnarray}
The following lemma will be useful in our convergence proof.
\begin{lemma} \cite{Antjd2}
\label{lemma2}
Let  \assref{assumption2} be fulfilled. Then the following estimate holds
{\small
\begin{eqnarray}
\label{comp1}
\left\Vert\left(\prod_{j=l}^me^{\Delta tA_{h,j}}\right)(-A_{h,l})^{\gamma}\right\Vert_{L(H)}&\leq& Ct_{m-l}^{-\gamma}, \quad 0\leq l< m,\quad 0\leq \gamma<1,\\
\label{comp1a}
\left\Vert(-A_{h,k})^{\gamma_1}\left(\prod_{j=l}^me^{\Delta tA_{h,j}}\right)(-A_{h,l})^{-\gamma_2}\right\Vert_{L(H)}&\leq& Ct_{m-l}^{\gamma_2-\gamma_1}, \quad 0\leq l< m,
\end{eqnarray}
}
$0\leq \gamma_1\leq 1$, $ 0<\gamma_2\leq 1$, where $C$ is a positive constant independent of $m$, $l$, $h$ and $\Delta t$.
\end{lemma}

\begin{lemma}
\label{smoothing1}
\begin{itemize}
\item[(i)] For all $\alpha\geq0$, the following estimate holds
\begin{eqnarray}
\left\Vert R^h(t,s)(-A_h(s))^{\alpha}\right\Vert_{L(H)}\leq C(t-s)^{-\alpha}, \quad t,s\in[0, T].
\end{eqnarray}
\item[(ii)] For all $\alpha\in[0,1]$, the following estimate holds
\begin{eqnarray}
\left\Vert\left(U_h(t_j, t_{j-1})-e^{\Delta tA_{h,j-1}}\right)\left(-A_{h,j-1}\right)^{-\alpha}\right\Vert_{L(H)}\leq C\Delta t^{1+\alpha}.
\end{eqnarray}
\item[(iii)] For all $\alpha\in[0,1)$, the following estimate holds
\begin{eqnarray}
\left\Vert\left(U_h(t_j, t_{j-1})-e^{\Delta tA_{h,j-1}}\right)\left(-A_{h,j-1}\right)^{\alpha}\right\Vert_{L(H)}\leq C\Delta t^{1-\alpha}.
\end{eqnarray}
\item[(iv)] For all $\alpha\in[0,1]$, the following estimate holds
\begin{eqnarray}
\left\Vert (-A_{h,j-1})^{-\alpha}\left(U_h(t_j,t_{j-1})-e^{\Delta tA_{h,j-1}}\right)\right\Vert_{L(H)}\leq C\Delta t^{1+\alpha}.
\end{eqnarray}
\end{itemize}
\end{lemma}
\begin{proof}
From the integral equation \eqref{ref7}, we have
{\small
\begin{eqnarray}
\label{smoothing2}
R^h(t,s)(-A_h(s))^{\alpha}=e^{A_h(s)(t-s)}(-A_h(s))^{\alpha}+\int_s^tR^h_1(t,\tau)R^h(\tau,s)(-A_h(s))^{\alpha}d\tau.
\end{eqnarray}
}
Taking the norm in both sides of \eqref{smoothing2}, using \eqref{smooth1} and \lemref{evolutionlemma} yields
{\small
\begin{eqnarray}
\label{smoothing3}
\left\Vert R^h(t,s)(-A_h(s))^{\alpha}\right\Vert_{L(H)}&\leq&\left\Vert e^{A_h(s)(t-s)}(-A_h(s))^{\alpha}\right\Vert_{L(H)}\nonumber\\
&+&\int_s^t\Vert R^h_1(\tau,s)\Vert_{L(H)}\left\Vert R^h(\tau,s)(-A_h(s))^{\alpha}\right\Vert_{L(H)}d\tau\nonumber\\
&\leq& C(t-s)^{-\alpha}+C\int_s^t\left\Vert R^h(\tau,s)(-A_h(s))^{\alpha}\right\Vert_{L(H)}d\tau.
\end{eqnarray}
}
Applying the continuous Gronwall's lemma to \eqref{smoothing3} yields
\begin{eqnarray}
\left\Vert R^h(t,s)(-A_h(s))^{\alpha}\right\Vert_{L(H)}\leq C(t-s)^{-\alpha}.
\end{eqnarray}
This completes the proof of (i).  
From \eqref{ref6} and \eqref{ref7}, we have
\begin{eqnarray}
\label{fonda4}
U_h(t_j, t_{j-1})-e^{\Delta tA_{h,j-1}}&=&\int_{t_{j-1}}^{t_j}e^{(t_j-\tau)A_h(\tau)}R_h(\tau, t_{j-1})d\tau\nonumber\\
&=&\int_{t_{j-1}}^{t_j}e^{(t_j-\tau)A_h(\tau)}R^h_1(\tau,t_{j-1})d\tau \nonumber\\
&+&\int_{t_{j-1}}^{t_j}e^{(t_j-\tau)A_h(\tau)}\left[\int_{t_{j-1}}^{\tau}R^h_1(\tau,s)
R^h(s,t_{j-1})ds\right]d\tau\nonumber\\
&=&\int_{t_{j-1}}^{t_j}e^{(t_j-\tau)A_h(\tau)}\left(A_h(\tau)-A_h(t_{j-1})\right)e^{A_{h,j-1}(\tau-t_{j-1})}d\tau \nonumber\\
&+&\int_{t_{j-1}}^{t_j}e^{(t_j-\tau)A_h(\tau)}\left[\int_{t_{j-1}}^{\tau}R^h_1(\tau,s)
R^h(s,t_{j-1})ds\right]d\tau.
\end{eqnarray}
Therefore, from \eqref{fonda4}, for all $\alpha\in[0,1]$, using \eqref{smooth1} and \lemref{evolutionlemma}, it holds that
\begin{eqnarray}
\label{fonda5}
&&\left\Vert\left(U_h(t_j, t_{j-1})-e^{\Delta tA_{h,j-1}}\right)\left(-A_{h,j-1}\right)^{-\alpha}\right\Vert_{L(H)}\nonumber\\
&\leq& \int_{t_{j-1}}^{t_j}\left\Vert e^{(t_j-\tau)A_h(\tau)}\left(A_h(\tau)-A_h(t_{j-1})\right)\left(-A_{h,j-1}\right)^{-1}\right.\nonumber\\
&&\left..e^{A_{h,j-1}(\tau-t_{j-1})}\left(-A_{h,j-1}\right)^{1-\alpha}\right\Vert_{L(H)}d\tau\nonumber\\
&+&\int_{t_{j-1}}^{t_j}\left\Vert e^{(t_j-\tau)A_h(\tau)}\right\Vert_{L(H)}\left[\int_{t_{j-1}}^{\tau}\Vert R^h_1(\tau, s)R^h(s, t_{j-1})\Vert_{L(H)}ds\right]d\tau\nonumber\\
&\leq& \int_{t_{j-1}}^{t_j}\left\Vert e^{(t_j-\tau)A_h(\tau)}\right\Vert_{L(H)}\left\Vert\left(A_h(\tau)-A_h(t_{j-1})\right)\left(-A_{h,j-1}\right)^{-1}\right\Vert_{L(H)}\nonumber\\
&\times&\left\Vert e^{A_{h,j-1}(\tau-t_{j-1})}\left(-A_{h,j-1}\right)^{1-\alpha}\right\Vert_{L(H)}d\tau\nonumber\\
&+&C\int_{t_{j-1}}^{t_j}\int_{t_{j-1}}^{\tau}dsd\tau\nonumber\\
&\leq& C\int_{t_{j-1}}^{t_j}(\tau-t_{j-1})^{\alpha}d\tau+C\Delta t^2\leq C\Delta t^{1+\alpha}.
\end{eqnarray}
This completes the proof of (ii). The proof of (iii) and (iv) are similar to that of (ii) using (i).
\end{proof}

The following lemma can be found in \cite{Stig2}
\begin{lemma}
\label{Stiglemma}
For all $\alpha_1, \alpha_2>0$ and $\alpha\in[0,1)$, there exist two positive constants $C_{\alpha_1,\alpha_2}$ and $C_{\alpha,\alpha_2}$ such that
\begin{eqnarray}
\label{eti1}
\Delta t\sum_{j=1}^mt_{m-j+1}^{-1+\alpha_1}t_j^{-1+\alpha_2}&\leq& C_{\alpha_1,\alpha_2}t_m^{-1+\alpha_1+\alpha_2},\\
\label{eti2}
\Delta t\sum_{j=1}^mt_{m-j+1}^{-\alpha}t_j^{-1+\alpha_2}&\leq& C_{\alpha,\alpha_2}t_m^{-\alpha+\alpha_2}.
\end{eqnarray}
\end{lemma}
\begin{proof}
The proof of \eqref{eti1} follows from the comparison with the integral
\begin{eqnarray}
\int_0^t(t-s)^{-1+\alpha_1}s^{-1+\alpha_2}ds.
\end{eqnarray}
The proof of \eqref{eti2} is a consequence of \eqref{eti1}.
\end{proof}

The following lemma is fundamental in our convergence analysis.
\begin{lemma}
\label{fonda}
Let \assref{assumption2}  be fulfilled. Then for all $1\leq i\leq m\leq M$.
\begin{itemize}
\item[(i)] The following estimate holds
\begin{eqnarray}
\label{fonda0}
\left\Vert\left(\prod_{j=i}^mU_h(t_j,t_{j-1})\right)-\left(\prod_{j=i-1}^{m-1}
e^{\Delta tA_{h,j}}\right)\right\Vert_{L(H)}\leq C\Delta t^{1-\epsilon},
\end{eqnarray}
where $\epsilon>0$ is a positive number small enough.
\item[(ii)]
The following estimate also holds
{\small
\begin{eqnarray}
\label{fonda00}
\left\Vert\left[\left(\prod_{j=i}^mU_h(t_j,t_{j-1})\right)-\left(\prod_{j=i-1}^{m-1}
e^{\Delta tA_{h,j}}\right)\right](-A_{h,i-1})^{-\epsilon}\right\Vert_{L(H)}\leq C\Delta t.
\end{eqnarray}
}
\end{itemize}
\end{lemma}
\begin{proof} 
 First of all note that
 {\small
\begin{eqnarray}
\label{idem}
\left(\prod_{j=i}^mU_h(t_j,t_{j-1})\right)-\left(\prod_{j=i-1}^{m-1}
e^{\Delta tA_{h,j}}\right)=\left(\prod_{j=i}^mU_h(t_j,t_{j-1})\right)-\left(\prod_{j=i}^{m}
e^{\Delta tA_{h,j-1}}\right).
\end{eqnarray}
}
Using the telescopic sum, \eqref{idem} can be rewritten as follows
\begin{eqnarray}
\label{fonda1}
&&\left(\prod_{j=i}^mU_h(t_j,t_{j-1})\right)-\left(\prod_{j=i}^{m}
e^{\Delta tA_{h,j-1}}\right)\nonumber\\
&=&\sum_{k=1}^{m-i+1}\left(\prod_{j=i+k}^mU_h(t_j,t_{j-1})\right)\left(U_h\left(t_{i+k-1}, t_{i+k-2}\right)-e^{\Delta tA_{h,i+k-2}}\right)\nonumber\\
&&.\left(\prod_{j=i}^{i+k-2}e^{\Delta tA_{h,j-1}}\right).
\end{eqnarray}
Writing down explicitly the first term of \eqref{fonda1} gives
\begin{eqnarray}
\label{fonda2}
&&\left(\prod_{j=i}^mU_h(t_j,t_{j-1})\right)-\left(\prod_{j=i}^{m}
e^{\Delta tA_{h,j-1}}\right)\nonumber\\
&=&\left(\prod_{j=i+1}^mU_h(t_j, t_{j-1})\right)\left(U_h(t_i,t_{i-1})-e^{\Delta tA_{h,i-1}}\right)\nonumber\\
&+&\sum_{k=2}^{m-i+1}\left(\prod_{j=i+k}^mU_h(t_j,t_{j-1})\right)\left(U_h\left(t_{i+k-1}, t_{i+k-2}\right)-e^{\Delta tA_{h,i+k-2}}\right)\nonumber\\
&&.\left(\prod_{j=i}^{i+k-2}e^{\Delta tA_{h,j-1}}\right).
\end{eqnarray}
Taking the norm in both sides of \eqref{fonda2}, using  \lemref{evolutionlemma}, \lemref{smoothing1} (ii) and \lemref{lemma2} yields
{\small
\begin{eqnarray}
\label{fonda6}
&&\left\Vert\left(\prod_{j=i}^mU_h(t_j,t_{j-1})\right)-\left(\prod_{j=i}^{m}
e^{\Delta tA_{h,j-1}}\right)\right\Vert_{L(H)}\nonumber\\
&\leq& \left\Vert U_h(t_{m-i+1},t_{i})\right\Vert_{L(H)}\left\Vert U_h(t_{i},t_{i-1})-e^{\Delta tA_{h,i-1}}\right\Vert_{L(H)}\nonumber\\
&+&\sum_{k=2}^{m-i+1}\left\Vert U_h(t_{m}, t_{i+k-1})\right\Vert_{L(H)}\left\Vert\left(U_h(t_{i+k-1}, t_{i+k-2})-e^{\Delta tA_{h,i+k-2}}\right)\left(-A_{h,i+k-2}\right)^{-1+\epsilon}\right\Vert_{L(H)}\nonumber\\
&\times&\left\Vert\left(-A_{h,i+k-2}\right)^{1-\epsilon}\left(\prod_{j=i}^{i+k-2}e^{\Delta tA_{h,j-1}}\right)\right\Vert_{L(H)}\nonumber\\
&\leq& C\Delta t+C\sum_{k=2}^{m-i+1}\Delta t^{2-\epsilon}t_{k-1}^{-1+\epsilon}\nonumber\\
&\leq& C\Delta t^{1-\epsilon}.
\end{eqnarray}
}
This completes the proof of (i). The proof of (ii) is similar to that of (i) using \eqref{comp1a} and \lemref{Stiglemma}. 
\end{proof}

With the above preparatory results in hand, we can now prove our main result.
\subsection{Proof of Theorem \ref{mainresult1}}
\label{preuve1}
Using triangle inequality, we split the fully discrete error in two parts as follows.
\begin{eqnarray}
\label{split}
\Vert X(t_m)-X^h_m\Vert_{L^2(\Omega, H)}&\leq& \Vert X(t_m)-X^h(t_m)\Vert_{L^2(\Omega,H)}+\Vert X^h(t_m)-X^h_m\Vert_{L^2(\Omega,H)}\nonumber\\
&=:& V+VI.
\end{eqnarray}
The space error $V$ is estimated in \lemref{spaceerrorlemma}. It remains to estimate the time error $VI$.  Note that the mild solution of \eqref{semi1} can be written as follows.
\begin{eqnarray}
\label{mild5}
X^h(t_m)&=&U_h(t_m,t_{m-1})X^h(t_{m-1})+\int_{t_{m-1}}^{t_m}U_h(t_m,s)P_hF\left(s,X^h(s)\right)ds\nonumber\\
&+&\int_{t_{m-1}}^{t_m}U_h(t_m,s)P_hB\left(s,X^h(s)\right)dW(s).
\end{eqnarray}
Iterating the mild solution \eqref{mild5} yields
{\small
\begin{eqnarray}
\label{mild6}
&&X^h(t_m)\nonumber\\
&=&\left(\prod_{j=1}^mU_h(t_j, t_{j-1})\right)P_hX_0+\int_{t_{m-1}}^{t_m}U_h(t_m,s)P_hF\left(s,X^h(s)\right)ds\nonumber\\
&+&\int_{t_{m-1}}^{t_m}U_h(t_m,s)P_hB\left(s,X^h(s)\right)dW(s)\nonumber\\
&+&\sum_{k=1}^{m-1}\int_{t_{m-k-1}}^{t_{m-k}}\left(\prod_{j=m-k+1}^{m}U_h(t_j,t_{j-1})\right)U_h(t_{m-k},s)P_hF\left(s, X^h(s)\right)ds\nonumber\\
&+&\sum_{k=1}^{m-1}\int_{t_{m-k-1}}^{t_{m-k}}\left(\prod_{j=m-k+1}^{m}U_h(t_j,t_{j-1})\right)U_h(t_{m-k},s)P_hB\left(s,X^h(s)\right)dW(s).
\end{eqnarray}
}
Iterating the numerical scheme \eqref{scheme4} by  substituting $X^h_j$, $j=m-1, \cdots,1$ only in the first term of \eqref{scheme4} by their expressions yields
{\small
\begin{eqnarray}
\label{num2}
X^h_m&=&\left(\prod_{j=0}^{m-1}e^{\Delta tA_{h,j}}\right)X^h_0+\int_{t_{m-1}}^{t_m}e^{(t_m-s)A_{h,m-1}}P_hF\left(t_{m-1}, X^h_{m-1}\right)ds\nonumber\\
&+&\int_{t_{m-1}}^{t_m}e^{\Delta tA_{h,m-1}}P_hB\left(t_{m-1},X^h_{m-1}\right)dW(s)\nonumber\\
&+&\sum_{k=1}^{m-1}\int_{t_{m-k-1}}^{t_{m-k}}\left(\prod_{j=m-k}^{m-1}e^{\Delta tA_{h,j}}\right)e^{(t_{m-k}-s)A_{h,m-k-1}}P_hF\left(t_{m-k-1},X^h_{m-k-1}\right)ds\nonumber\\
&+&\sum_{k=1}^{m-1}\int_{t_{m-k-1}}^{t_{m-k}}\left(\prod_{j=m-k}^{m-1}e^{\Delta tA_{h,j}}\right)e^{\Delta tA_{h,m-k-1}}P_hB\left(t_{m-k-1}, X^h_{m-k-1}\right)dW(s).
\end{eqnarray}
}
Substracting \eqref{num2} from \eqref{mild6} yields
{\small
\begin{eqnarray}
\label{refait1}
&&X^h(t_m)-X^h_m\nonumber\\
&=&\left(\prod_{j=1}^mU_h(t_j, t_{j-1})\right)P_hX_0-\left(\prod_{j=0}^{m-1}e^{\Delta tA_{h,j}}\right)P_hX_0\nonumber\\
&+&\int_{t_{m-1}}^{t_m}U_h(t_m,s)P_hF\left(s,X^h(s)\right)ds-\int_{t_{m-1}}^{t_m}e^{(t_m-s)A_{h,m-1}}P_hF\left(t_{m-1}, X^h_{m-1}\right)ds\nonumber\\
&+&\int_{t_{m-1}}^{t_m}U_h(t_m,s)P_hB\left(s,X^h(s)\right)dW(s)-\int_{t_{m-1}}^{t_m}e^{\Delta tA_{h,m-1}}P_hB\left(t_{m-1}, X^h_{m-1}\right)dW(s)\nonumber\\
&+&\sum_{k=1}^{m-1}\int_{t_{m-k-1}}^{t_{m-k}}\left(\prod_{j=m-k+1}^{m}U_h(t_j,t_{j-1})\right)U_h(t_{m-k},s)P_hF\left(s, X^h(s)\right)ds\nonumber\\
&&-\sum_{k=1}^{m-1}\int_{t_{m-k-1}}^{t_{m-k}}\left(\prod_{j=m-k}^{m-1}e^{\Delta tA_{h,j}}\right)e^{(t_{m-k}-s)A_{h,m-k-1}}P_hF\left(t_{m-k-1},X^h_{m-k-1}\right)ds\nonumber\\
&+&\sum_{k=1}^{m-1}\int_{t_{m-k-1}}^{t_{m-k}}\left(\prod_{j=m-k+1}^{m}U_h(t_j,t_{j-1})\right)U_h(t_{m-k},s)P_hB\left(s,X^h(s)\right)dW(s)\nonumber\\
&&-\sum_{k=1}^{m-1}\int_{t_{m-k-1}}^{t_{m-k}}\left(\prod_{j=m-k}^{m-1}e^{\Delta tA_{h,j}}\right)e^{\Delta tA_{h,m-k-1}}P_hB\left(t_{m-k-1}, X^h_{m-k-1}\right)dW(s)\nonumber\\
&=:&VI_1+VI_2+VI_3+VI_4+VI_5.\nonumber\\
\end{eqnarray}
}
Taking the norm in both sides of \eqref{refait1} yields 
\begin{eqnarray}
\label{refait2}
\Vert X^h(t_m)-X^h_m\Vert^2_{L^2(\Omega,H)}\leq 25\sum_{i=1}^5\Vert VI_i\Vert^2_{L^2(\Omega,H)}.
\end{eqnarray}
In what follows, we estimate separately $\Vert VI_i\Vert_{L^2(\Omega,H)}$, $i=1,\cdots, 5$.
 \subsubsection{Estimate of $VI_1$, $VI_2$ and $VI_3$}
 Using \lemref{fonda}, it holds that
 \begin{eqnarray}
 \label{multi1}
 \Vert VI_1\Vert_{L^2(\Omega,H)}&\leq& \left\Vert \left(\prod_{j=1}^mU_h(t_j, t_{j-1})\right)-\left(\prod_{j=0}^{m-1}e^{\Delta tA_{h,j}}\right)\right\Vert_{L(H)}\Vert X_0\Vert_{L^2(\Omega,H)}\nonumber\\
 &\leq& C\Delta t^{1-\epsilon}.
 \end{eqnarray}
 Using triangle inequality, \eqref{smooth2}, \lemref{evolutionlemma}, \assref{assumption3} and \thmref{theorem1}, it holds that
 {\small
 \begin{eqnarray}
 \label{multi2}
 \Vert VI_2\Vert_{L^2(\Omega,H)}&\leq& \int_{t_{m-1}}^{t_m}\left\Vert U_h(t_m,s)P_hF\left(s,X^h(s)\right)\right\Vert_{L^2(\Omega,H)}ds\nonumber\\
 &+&\int_{t_{m-1}}^{t_m}\left\Vert e^{(t_m-s)A_{h,m-1}}\left[P_hF\left(t_{m-1},X^h_{m-1}\right)-P_hF\left(t_{m-1}, X^h(t_{m-1})\right)\right]\right\Vert_{L^2(\Omega,H)}ds\nonumber\\
 &+&\int_{t_{m-1}}^{t_m}\left\Vert e^{(t_m-s)A_{h,m-1}}P_hF\left(t_{m-1}, X^h(t_{m-1})\right)\right\Vert_{L^2(\Omega,H)}ds\nonumber\\
 &\leq& C\int_{t_{m-1}}^{t_m}ds+C\int_{t_{m-1}}^{t_m}\Vert X^h(t_{m-1})-X^h_{m-1}\Vert_{L^2(\Omega,H)}ds+C\int_{t_{m-1}}^{t_m}ds\nonumber\\
 &\leq& C\Delta t+C\Delta t\Vert X^h(t_{m-1})-X^h_{m-1}\Vert_{L^2(\Omega,H)}.
 \end{eqnarray}
 }
Applying the It\^{o}-isometry property, using \assref{assumption4}, \eqref{smooth2}, \thmref{theorem1} and \lemref{evolutionlemma} yields 
{\small
 \begin{eqnarray}
 \label{multi3}
 \Vert VI_3\Vert^2_{L^2(\Omega,H)}&\leq& 9\int_{t_{m-1}}^{t_m}\mathbb{E}\left\Vert U_h(t_m,s)P_hB\left(s,X^h(s)\right)\right\Vert^2_{L^0_2}ds\nonumber\\
 &+&9\int_{t_{m-1}}^{t_m}\mathbb{E}\left\Vert e^{\Delta tA_{h,m-1}}\left[P_hB\left(t_{m-1},X^h_{m-1}\right)-P_hB\left(t_{m-1}, X^h(t_{m-1})\right)\right]\right\Vert^2_{L^0_2}ds\nonumber\\
 &+&9\int_{t_{m-1}}^{t_m}\mathbb{E}\left\Vert e^{\Delta tA_{h,m-1}}P_hF\left(t_{m-1}, X^h(t_{m-1})\right)\right\Vert^2_{L^0_2}ds\nonumber\\
 &\leq& C\int_{t_{m-1}}^{t_m}ds+C\int_{t_{m-1}}^{t_m}\Vert X^h(t_{m-1})-X^h_{m-1}\Vert^2_{L^2(\Omega,H)}ds+C\int_{t_{m-1}}^{t_m}ds\nonumber\\
 &\leq& C\Delta t+C\Delta t\Vert X^h(t_{m-1})-X^h_{m-1}\Vert^2_{L^2(\Omega,H)}.
 \end{eqnarray}
 }
\subsubsection{Estimate of $VI_4$} To estimate $VI_4$, we split it in five terms as follows.
{\small
\begin{eqnarray}
\label{eza1}
&&VI_{4}\nonumber\\
&=&\sum_{k=1}^{m-1}\int_{t_{m-k-1}}^{t_{m-k}}\left(\prod_{j=m-k+1}^mU_h(t_j,t_{j-1})\right)U_h(t_{m-k},s)\left[P_hF\left(s,X^h(s)\right)-P_hF\left(t_{m-k-1},X^h(t_{m-k-1})\right)\right]ds\nonumber\\
&+&\sum_{k=1}^{m-1}\int_{t_{m-k-1}}^{t_{m-k}}\left(\prod_{j=m-k+1}^mU_h(t_j,t_{j-1})\right)\left[U_h(t_{m-k},s)-U_h(t_{m-k}, t_{m-k-1})\right]P_hF\left(t_{m-k-1},X^h(t_{m-k-1})\right)ds\nonumber\\
&+&\sum_{k=1}^{m-1}\int_{t_{m-k-1}}^{t_{m-k}}\left[\left(\prod_{j=m-k}^mU_h(t_j,t_{j-1})\right)-\left(\prod_{j=m-k-1}^{m-1}e^{\Delta tA_{h,j}}\right)\right]P_hF\left(t_{m-k-1},X^h(t_{m-k-1})\right)ds\nonumber\\
&+&\sum_{k=1}^{m-1}\int_{t_{m-k-1}}^{t_{m-k}}\left(\prod_{j=m-k}^{m-1}e^{\Delta tA_{h,j}}\right)\left(e^{\Delta tA_{h,m-k-1}}-e^{(t_{m-k}-s)A_{h,m-k-1}}\right)P_hF\left(t_{m-k-1},X^h(t_{m-k-1})\right)ds\nonumber\\
&+&\sum_{k=1}^{m-1}\int_{t_{m-k-1}}^{t_{m-k}}\left(\prod_{j=m-k}^{m-1}e^{\Delta tA_{h,j}}\right)e^{(t_{m-k}-s)A_{h,m-k-1}}\left[P_hF\left(t_{m-k-1},X^h(t_{m-k-1})\right)-P_hF\left(t_{m-k-1},X^h_{m-k-1}\right)\right]ds\nonumber\\
&=:&VI_{41}+VI_{42}+VI_{43}+VI_{44}+VI_{45}.
\end{eqnarray}
}
Using \lemref{evolutionlemma}, \assref{assumption3} and \lemref{regularitylemma} yields
{\small
\begin{eqnarray}
\label{eza2}
&&\Vert VI_{41}\Vert_{L^2(\Omega,H)}\nonumber\\
&\leq& C\sum_{k=1}^{m-1}\int_{t_{m-k-1}}^{t_{m-k}}\left\Vert P_hF\left(s,X^h(s)\right)-P_hF\left(t_{m-k-1},X^h(t_{m-k-1})\right)\right\Vert_{L^2(\Omega,H)}ds\nonumber\\
&\leq&C\sum_{k=1}^{m-1}\int_{t_{m-k-1}}^{t_{m-k}}(s-t_{m-k-1})^{\beta/2}ds+C\sum_{k=1}^{m-1}\int_{t_{m-k-1}}^{t_{m-k}}\Vert X^h(s)-X^h(t_{m-k-1})\Vert_{L^2(\Omega,H)}ds\nonumber\\
&\leq& C\Delta t^{\beta/2}+\sum_{k=1}^{m-1}\int_{t_{m-k-1}}^{t_{m-k}}(s-t_{m-k-1})^{\min(\beta,1)/2}ds\nonumber\\
&\leq& C\Delta t^{\min(\beta,1)/2}.
\end{eqnarray}
}
Using \lemref{evolutionlemma},  \assref{assumption3} and \thmref{theorem1} gives
{\small
\begin{eqnarray}
\label{eza3}
&&\Vert VI_{42}\Vert_{L^2(\Omega,H)}\nonumber\\
&\leq&C\sum_{k=1}^{m-1}\int_{t_{m-k-1}}^{t_{m-k}}\left\Vert U_h(t_m,t_{m-k})U_h(t_{m-k},s)(\mathbf{I}-U_h(s,t_{m-k-1})\right\Vert_{L(H)}\nonumber\\
&\times&\left\Vert P_hF\left(t_{m-k-1}, X^h(t_{m-k-1})\right)\right\Vert_{L^2(\Omega,H)}ds\nonumber\\
&\leq& C\sum_{k=1}^{m-1}\int_{t_{m-k-1}}^{t_{m-k}}\left\Vert U_h(t_m,t_{m-k})(-A_{h,m-k})^{1-\epsilon}\right\Vert_{L(H)}\left\Vert (-A_{h,m-k})^{-1+\epsilon} U_h(t_{m-k},s)(-A_{h,m-k})^{1-\epsilon}\right\Vert_{L(H)}\nonumber\\
&\times&\left\Vert(-A_{h,m-k})^{-1+\epsilon}\left(\mathbf{I}-U_h(s,t_{m-k-1})\right)\right\Vert_{L(H)}ds\nonumber\\
&\leq& C\sum_{k=1}^{m-1}\int_{t_{m-k-1}}^{t_{m-k}}(t_m-t_{m-k})^{-1+\epsilon}(s-t_{m-k-1})^{1-\epsilon}ds\nonumber\\
&\leq& C\Delta t^{1-\epsilon}\sum_{k=1}^{m-1}\int_{t_{m-k-1}}^{t_{m-k}}t_k^{-1+\epsilon}ds\nonumber\\
&\leq&C\Delta t^{1-\epsilon}\sum_{k=1}^{m-1}\Delta tt_k^{-1+\epsilon}\nonumber\\
&\leq& C\Delta t^{1-\epsilon}.
\end{eqnarray}
}
Using  \lemref{lemma2}, \assref{assumption3}, \thmref{theorem1}, \eqref{smooth2} and \eqref{smooth1} yields
\begin{eqnarray}
\label{eza4}
&&\Vert VI_{43}\Vert_{L^2(\Omega,H)}\nonumber\\
&\leq& \sum_{k=1}^{m-1}\int_{t_{m-k-1}}^{t_{m-k}}\left\Vert\left(\prod_{j=m-k}^{m-1}e^{\Delta tA_{h,j}}\right)\left(e^{(s-t_{m-k-1})A_{h,m-k-1}}-\mathbf{I}\right)e^{(t_{m-k}-s)A_{h,m-k-1}}\right\Vert_{L(H)}\nonumber\\
&\times& \left\Vert P_hF\left(t_{m-k-1},X^h(t_{m-k-1})\right)\right\Vert_{L^2(\Omega,H)}ds\nonumber\\
&\leq& C\sum_{k=1}^{m-1}\int_{t_{m-k-1}}^{t_{m-k}}\left\Vert\left(\prod_{j=m-k}^{m-1}e^{\Delta tA_{h,j}}\right)\left(-A_{h,m-k-1}\right)^{1-\epsilon}\right\Vert_{L(H)}\nonumber\\
&\times&\left\Vert\left(-A_{h,m-k-1}\right)^{-1+\epsilon}\left(e^{(s-t_{m-k-1})A_{h,m-k-1}}-\mathbf{I}\right)\right\Vert_{L(H)}\left\Vert e^{(t_{m-k}-s)A_{h,m-k-1}}\right\Vert_{L(H)}ds\nonumber\\
&\leq& C\sum_{k=1}^{m-1}\int_{t_{m-k-1}}^{t_{m-k}}t_k^{-1+\epsilon}(s-t_{m-k-1})^{1-\epsilon}ds\nonumber\\
&\leq& C\Delta t^{1-\epsilon}\sum_{k=1}^{m-1}\int_{t_{m-k-1}}^{t_{m-k}}t_k^{-1+\epsilon}\Delta t\nonumber\\
&\leq& C\Delta t^{1-\epsilon}.
\end{eqnarray}
Using \lemref{lemma2}, \eqref{smooth2}, \eqref{smooth1}, \assref{assumption3} and \lemref{evolutionlemma} yields
\begin{eqnarray}
\label{eza5}
&&\Vert VI_{44}\Vert_{L^2(\Omega,H)}\nonumber\\
&\leq& \sum_{k=1}^{m-1}\int_{t_{m-k-1}}^{t_{m-k}}\left\Vert\left(\prod_{j=m-k}^{m-1}e^{\Delta tA_{h,j}}\right)\left(\mathbf{I}-e^{(s-t_{m-k-1})A_{h,m-k-1}}\right)e^{(t_{m-k}-s)A_{h,m-k-1}}\right\Vert_{L(H)}\nonumber\\
&\times& \left\Vert P_hF\left(t_{m-k-1}, X^h(t_{m-k-1})\right)\right\Vert_{L^2(\Omega,H)}ds\nonumber\\
&\leq&C \sum_{k=1}^{m-1}\int_{t_{m-k-1}}^{t_{m-k}}\left\Vert\left(\prod_{j=m-k}^{m-1}e^{\Delta tA_{h,j}}\right)\left(-A_{h,m-k}\right)^{1-\epsilon}\right\Vert_{L(H)}\nonumber\\
&\times&\left\Vert\left(-A_{h,m-k}\right)^{-1+\epsilon}\left(\mathbf{I}-e^{(s-t_{m-k-1})A_{h,m-k-1}}\right)\right\Vert_{L(H)}\left\Vert e^{(t_{m-k}-s)A_{h,m-k-1}}\right\Vert_{L(H)}\nonumber\\
&\leq& C\sum_{k=1}^{m-1}\int_{t_{m-k-1}}^{t_{m-k}}t_k^{-1+\epsilon}(s-t_{m-k-1})^{1-\epsilon}ds\nonumber\\
&\leq& C\Delta t^{1-\epsilon}\sum_{k=1}^{m-1}\int_{t_{m-k-1}}^{t_{m-k}}t_k^{-1+\epsilon}ds\nonumber\\
&\leq& C\Delta t^{1-\epsilon}.
\end{eqnarray}
Using  \lemref{lemma2} and \assref{assumption3} yields
\begin{eqnarray}
\label{eza6}
\Vert VI_{45}\Vert_{L^2(\Omega,H)}&\leq& C\sum_{k=1}^{m-1}\int_{t_{m-k-1}}^{t_{m-k}}\Vert X^h\left(t_{m-k-1}\right)-X^h_{m-k-1}\Vert_{L^2(\Omega, H)}\nonumber\\
&\leq& C\Delta t\sum_{k=0}^{m-1}\Vert X^h(t_k)-X^h_k\Vert_{L^2(\Omega,H)}.
\end{eqnarray}
Substituting \eqref{eza6}, \eqref{eza5}, \eqref{eza4}, \eqref{eza3} and \eqref{eza2} in \eqref{eza1} yields
\begin{eqnarray}
\label{multi4}
\Vert VI_4\Vert_{L^2(\Omega,H)}\leq C\Delta t^{\min(\beta,1)/2}+C\Delta t\sum_{k=0}^{m-1}\Vert X^h(t_k)-X^h_k\Vert_{L^2(\Omega,H)}.
\end{eqnarray}
\subsubsection{Estimate of $VI_5$}
To estimate $VI_5$, we split it in four terms as follows
{\small
\begin{eqnarray}
\label{boto1}
&&VI_{5}\nonumber\\
&=&\sum_{k=1}^{m-1}\int_{t_{m-k-1}}^{t_{m-k}}\left(\prod_{j=m-k+1}^mU_h(t_j,t_{j-1})\right)U_h(t_{m-k},s)\left[P_hB\left(s,X^h(s)\right)-P_hB\left(t_{m-k-1},X^h(t_{m-k-1})\right)\right]dW(s)\nonumber\\
&+&\sum_{k=1}^{m-1}\int_{t_{m-k-1}}^{t_{m-k}}\left(\prod_{j=m-k+1}^mU_h(t_j,t_{j-1})\right)\left[U_h(t_{m-k},s)-U_h(t_{m-k}, t_{m-k-1})\right]P_hB\left(t_{m-k-1},X^h(t_{m-k-1})\right)dW(s)\nonumber\\
&+&\sum_{k=1}^{m-1}\int_{t_{m-k-1}}^{t_{m-k}}\left[\left(\prod_{j=m-k}^mU_h(t_j,t_{j-1})\right)-\left(\prod_{j=m-k-1}^{m-1}e^{\Delta tA_{h,j}}\right)\right]P_hB\left(t_{m-k-1},X^h(t_{m-k-1})\right)dW(s)\nonumber\\
&+&\sum_{k=1}^{m-1}\int_{t_{m-k-1}}^{t_{m-k}}\left(\prod_{j=m-k-1}^{m-1}e^{\Delta tA_{h,j}}\right)\left[P_hB\left(t_{m-k-1},X^h(t_{m-k-1})\right)-P_hB\left(t_{m-k-1},X^h_{m-k-1}\right)\right]dW(s)\nonumber\\
&=:&VI_{51}+VI_{52}+VI_{53}+VI_{54}.
\end{eqnarray}
}
Using the It\^{o}-isometry property, \lemref{evolutionlemma}, \assref{assumption4} and \lemref{regularitylemma} yields
\begin{eqnarray}
\label{boto2}
&&\Vert VI_{51}\Vert^2_{L^2(\Omega,H)}\nonumber\\
&=&\sum_{k=1}^{m-1}\int_{t_{m-k-1}}^{t_{m-k}}\mathbb{E}\left\Vert U_h(t_m,s)\left[P_hB\left(s,X^h(s)\right)-P_hB\left(t_{m-k-1}, X^h(t_{m-k-1})\right)\right]\right\Vert^2_{L^0_2}ds\nonumber\\
&\leq&C\sum_{k=1}^{m-1}\int_{t_{m-k-1}}^{t_{m-k}}(s-t_{m-k-1})^{\beta}ds+C\sum_{k=1}^{m-1}\int_{t_{m-k-1}}^{t_{m-k}}\left\Vert X^h(s)-X^h(t_{m-k-1})\right\Vert^2_{L^2(\Omega,H)}ds\nonumber\\
&\leq& C\Delta t^{\beta}+C\sum_{k=1}^{m-1}\int_{t_{m-k-1}}^{t_{m-k}}(s-t_{m-k-1})^{\min(\beta,1)}ds\nonumber\\
&\leq& C\Delta t^{\min(\beta,1)}.
\end{eqnarray}
Applying the It\^{o}-isometry property, using \lemref{evolutionlemma}, \assref{assumption4} and \lemref{regularitylemma} yields
{\small
\begin{eqnarray}
\label{boto3}
&&\Vert VI_{52}\Vert^2_{L^2(\Omega,H)}\nonumber\\
&=&\sum_{k=1}^{m-1}\int_{t_{m-k-1}}^{t_{m-k}}\mathbb{E}\left\Vert U_h(t_m, t_{m-k})U_h(t_{m-k},s)\left(\mathbf{I}-U_h(s, t_{m-k-1})\right)P_hB\left(t_{m-k-1}, X^h(t_{m-k-1})\right)\right\Vert^2_{L^0_2} ds\nonumber\\
&\leq&C\sum_{k=1}^{m-1}\int_{t_{m-k-1}}^{t_{m-k}}\left\Vert U_h(t_m, t_{m-k})\left(-A_{h,m-k}\right)^{\frac{1-\epsilon}{2}}\right\Vert^2_{L(H)}\Vert (-A_{m-k})^{\frac{-1+\epsilon}{2}} U_h(t_{m-k},s)(-A_{h,m-k})^{\frac{1-\epsilon}{2}}\Vert^2_{L(H)}\nonumber\\
&\times&\left\Vert\left(-A_{h,m-k}\right)^{\frac{-1+\epsilon}{2}}\left(\mathbf{I}-U_h(s,t_{m-k-1})\right)\right\Vert^2_{L(H)}ds\nonumber\\
&\leq& C\sum_{k=1}^{m-1}\int_{t_{m-k-1}}^{t_{m-k}}t_k^{-1+\epsilon}(s-t_{m-k-1})^{1-\epsilon}ds\nonumber\\
&\leq& C\Delta t^{1-\epsilon}\sum_{k=1}^{m-1}\int_{t_{m-k-1}}^{t_{m-k}}t_k^{-1+\epsilon}ds\nonumber\\
&\leq& C\Delta t^{1-\epsilon}.\nonumber\\
\end{eqnarray}
}
Applying the It\^{o}-isometry property, using \lemref{fonda}, \assref{assumption4} and \lemref{regularitylemma} yields
\begin{eqnarray}
\label{boto4}
\Vert VI_{53}\Vert^2_{L^2(\Omega,H)}&=&\sum_{k=1}^{m-1}\int_{t_{m-k-1}}^{t_{m-k}}\mathbb{E}\left\Vert\left[\left(\prod_{j=m-k}^mU_h(t_j, t_{j-1})\right)-\left(\prod_{j=m-k-1}^{m-1}e^{\Delta tA_{h,j}}\right)\right]\right.\nonumber\\
&&.\left.P_hB\left(t_{m-k-1}, X^h(t_{m-k-1})\right)\right\Vert^2_{L^0_2}ds\nonumber\\
&\leq& C\sum_{k=1}^{m-1}\int_{t_{m-k-1}}^{t_{m-k}}\Delta t^{1-\epsilon}ds\nonumber\\
&\leq& C\Delta t^{1-\epsilon}.
\end{eqnarray}
Applying the It\^{o}-isometry property, \lemref{lemma2} and \assref{assumption4} yields 
\begin{eqnarray}
\label{boto5}
\Vert VI_{54}\Vert^2_{L^2(\Omega,H)}&=&\sum_{k=1}^{m-1}\int_{t_{m-k-1}}^{t_{m-k}}\mathbb{E}\left\Vert\left(\prod_{j=m-k-1}^{m-1}e^{\Delta tA_{h,j}}\right)\right.\nonumber\\
&&.\left.\left[P_hB\left(t_{m-k-1}, X^h(t_{m-k-1})\right)-P_hB\left(t_{m-k-1}, X^h_{m-k-1}\right)\right]\right\Vert^2_{L^0_2}ds\nonumber\\
&\leq& C\sum_{k=1}^{m-1}\int_{t_{m-k-1}}^{t_{m-k}}\left\Vert X^h(t_{m-k-1})-X^h_{m-k-1}\right\Vert^2_{L^2(\Omega,H)}ds\nonumber\\
&\leq& C\Delta t\sum_{k=0}^{m-1}\Vert X^h(t_k)-X^h_k\Vert^2_{L^2(\Omega,H)}.
\end{eqnarray}
Substituting \eqref{boto5}, \eqref{boto4}, \eqref{boto3} and \eqref{boto2}  in \eqref{boto1} yields
\begin{eqnarray}
\label{multi5}
\Vert VI_5\Vert^2_{L^2(\Omega,H)}\leq C\Delta t^{\min(\beta,1)}+C\Delta t\sum_{k=0}^{m-1}\Vert X^h(t_k)-X^h_k\Vert^2_{L^2(\Omega,H)}.
\end{eqnarray}
Substituting \eqref{multi5}, \eqref{multi4}, \eqref{multi3}, \eqref{multi2} and \eqref{multi1} in \eqref{refait1} yields
\begin{eqnarray}
\label{multi6}
\Vert X^h(t_m)-X^h_m\Vert^2_{L^2(\Omega,H)}\leq C\Delta t^{\min(\beta,1-\epsilon)}+C\Delta t\sum_{k=0}^{m-1}\Vert X^h(t_k)-X^h_k\Vert^2_{L^2(\Omega,H)}.
\end{eqnarray}
Applying the discrete Gronwall's lemma to \eqref{multi6} yields
\begin{eqnarray}
\Vert X^h(t_m)-X^h_m\Vert_{L^2(\Omega,H)}\leq C\Delta t^{\min(\beta, 1-\epsilon)/2}.
\end{eqnarray}
Note that to achieve optimal convergence $1/2$ when $ \beta\geq 1$, we only need to re-estimate $\Vert VI_{52}\Vert_{L^2(\Omega,H)}$ and $\Vert VI_{53}\Vert_{L^2(\Omega,H)}$ 
by using \assref{assumption5} and \lemref{fonda} (ii). This is straightforward. The proof of \thmref{mainresult1} is therefore completed.
 
\section{Numerical experiments}
\label{experiment}
We consider the following  stochastic  reactive dominated advection
diffusion reaction  with  constant diagonal difussion tensor  
\begin{eqnarray}
\label{reactiondif1}
dX=\left[(1+e^{-t}) \left(\varDelta X-\nabla \cdot(\mathbf{q}X)\right)-\dfrac{e^{-t} X}{\vert X\vert +1}\right]dt+XdW,\quad X(0)=0,
\end{eqnarray}
with  mixed Neumann-Dirichlet boundary conditions on $\Lambda=[0,L_1]\times[0,L_2]$. 
The Dirichlet boundary condition is $X=1$ at $\Gamma=\{ (x,y) :\; x =0\}$ and 
we use the homogeneous Neumann boundary conditions elsewhere.
The eigenfunctions $ \{e_{i,j} \} =\{e_{i}^{(1)}\otimes e_{j}^{(2)}\}_{i,j\geq 0}
$ of  the covariance operator $Q$ are the same as for the Laplace operator $-\varDelta$  with homogeneous boundary condition,  given by 
\begin{eqnarray*}
e_{0}^{(l)}(x)=\sqrt{\dfrac{1}{L_{l}}},\qquad 
e_{i}^{(l)}(x)=\sqrt{\dfrac{2}{L_{l}}}\cos\left(\dfrac{i \pi }{L_{l}} x\right),
\, i \in \mathbb{N},
\end{eqnarray*}
where $l \in \left\lbrace 1, 2 \right\rbrace,\, x\in \Lambda$.
We assume that the noise can be represented as 
 \begin{eqnarray}
  \label{eq:W1}
  W(x,t)=\underset{(i, j) \in  \mathbb{N}^{2}}{\sum}\sqrt{\lambda_{i,j}}e_{i,j}(x)\beta_{i,j}(t), 
\end{eqnarray}
where $\beta_{i,j}(t)$ are
independent and identically distributed standard Brownian motions,  $\lambda_{i,j}$, $(i,j)\in \mathbb{N}^{2}$ are the eigenvalues  of $Q$, with
\begin{eqnarray}
\label{noise2}
 \lambda_{i,j}=\left( i^{2}+j^{2}\right)^{-(\beta +\delta)}, \, \beta>0,
\end{eqnarray} 
in the representation \eqref{eq:W1} for some small $\delta>0$. To obtain trace class noise, it is enough to have $\beta+\delta>1$. In our simulations, we take $\beta\in\{1.5, 2\}$ and $\delta=0.001$.
In \eqref{nemistekii1}, we take $b(x,u)=4u$, $x\in \Lambda$ and $u\in \mathbb{R}$. Therefore, from \cite[Section 4]{Arnulf1} it follows that the operators  $B$ defined by \eqref{nemistekii1}
fulfills   \assref{assumption4} and \assref{assumption5}.  The function $F$ is given by $F(t,v)= -\dfrac{e^{-t} v}{1+\vert v\vert}$, $t\in[0, T]$, $v\in H$ and obviously satisfies \assref{assumption3}. The nonlinear operator $A(t)$ is given by 
\begin{eqnarray}
A(t)=(1+e^{-t})\left(\varDelta(.)-\nabla.\mathbf{v}(.)\right),\quad t\in[0, T],
\end{eqnarray}
where $\mathbf{v}$ is the Darcy velocity.
We obtain the Darcy velocity field $\mathbf{v}=(q_i)$  by solving the following  system
\begin{equation}
  \label{couple1}
  \nabla \cdot\mathbf{v} =0, \qquad \mathbf{v}=-\mathbf{k} \nabla p,
\end{equation}
with  Dirichlet boundary conditions on 
$\Gamma_{D}^{1}=\left\lbrace 0,L_1 \right\rbrace \times \left[
  0,L_2\right] $ and Neumann boundary conditions on
$\Gamma_{N}^{1}=\left( 0,L_1\right)\times\left\lbrace 0,L_2\right\rbrace $ such that 
\begin{eqnarray*}
p&=&\left\lbrace \begin{array}{l}
1 \quad \text{in}\quad \left\lbrace 0 \right\rbrace \times\left[ 0,L_2\right]\\
0 \quad \text{in}\quad \left\lbrace L_1 \right\rbrace \times\left[ 0,L_2\right]
\end{array}\right. 
\end{eqnarray*}
and $- \mathbf{k} \,\nabla p (\mathbf{x},t)\,\cdot \mathbf{n} =0$ in  $\Gamma_{N}^{1}$. 
Here, we use a constant permeabily tensor $\mathbf{k}$ and have obtained almost a linear presure $p$.
Clearly $\mathcal{D}(A(t))=\mathcal{D}(A(0))$, $t\in[0, T]$ and $\mathcal{D}((-A(t))^{\alpha})=\mathcal{D}((-A(0))^{\alpha})$, $t\in[0, T]$, $0\leq \alpha\leq 1$. 
The  function $q_{ij}(x,t)$ defined in \eqref{family} is given by $q_{ii}(x,t)=1+e^{-t},$ and $q_{ij}(x,t)=0,\, i\neq j$. Since $q_{i i}(x,t)$ is bounded below by $1+e^{-T}$,
it follows that the ellipticity  condition \eqref{ellip} holds and therefore as a consequence of \secref{spacediscretization}, 
it follows that $A(t)$ is sectorial. Obviously \assref{assumption2} is fulfilled. 
\begin{figure}[!ht]
 \begin{center}
 \includegraphics[width=0.45\textwidth]{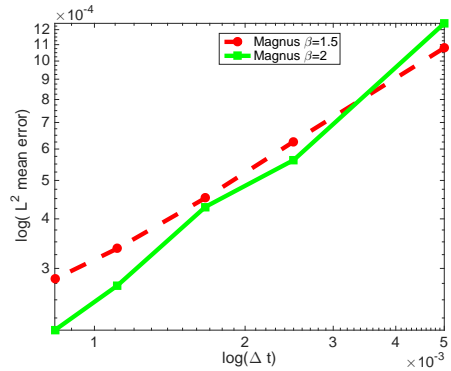}
  \end{center}
 \caption{Convergence of the  implicit scheme for $\beta=1$,  and $\beta=2$ in \eqref{noise2}.
 The order of convergence in time  is $0.57$  for $\beta=1$, $0.54$ for $\beta=2$. The total number of samples used is $100$.}
 \label{FIGII}
 \end{figure}
 
 In \figref{FIGII}, we can observe the convergence of the the stochastic Magnus  scheme for two noise's parameters.
 Indeed the order of convergence in time is $0.57$  for $\beta=1$ and $0.54$ for $\beta=2$. 
 These orders are close to the theoretical  orders $0.5$ obtained in \thmref{mainresult1} for $\beta=1$ and $\beta=2$.

\end{document}